\documentclass[12pt]{article}
\usepackage{amssymb,amsmath,amsthm,epsfig}
\usepackage{eepic}
\usepackage{epic}
\usepackage{graphicx}
\usepackage{color}
\usepackage{ifpdf}
\usepackage{amsthm,amsmath,amssymb,bm}
\usepackage{mathrsfs}
\usepackage{float}
\usepackage{dsfont}
\usepackage{graphicx} 
\usepackage{subfigure}
\usepackage{algorithm, algorithmic}
\usepackage[title]{appendix}
\usepackage{dsfont}

\usepackage{tabularx}
\usepackage{multirow}
\usepackage{makecell}
\usepackage{booktabs}
\usepackage{threeparttable}
\usepackage[numbers,sort&compress]{natbib}

\theoremstyle{plain}
\newtheorem{lem}{Lemma}[section]
\newtheorem{thm}[lem]{Theorem}

\newtheorem{prop}[lem]{Proposition}

\theoremstyle{definition}

\theoremstyle{remark}
\newtheorem{rem}{Remark}[section]

\begin{document}

\title{ \large\bf Perron-Frobenius operator filter for stochastic dynamical systems}

\author{Ningxin Liu\thanks{School of Mathematical Sciences,  Tongji University, Shanghai 200092, China. ({\tt nxliu@tongji.edu.cn}).}
\and
Lijian Jiang\thanks{School of Mathematical Sciences,  Tongji University, Shanghai 200092, China. ({\tt  ljjiang@tongji.edu.cn}).}
}

\date{}
\maketitle
\begin{center}{\bf Abstract}
\end{center}\smallskip

The filtering problems are derived from a sequential minimization of a quadratic function representing a compromise between model and data. In this paper, we use the Perron-Frobenius operator in stochastic process to develop a Perron-Frobenius operator filter. The proposed  method belongs to  Bayesian filtering and works for non-Gaussian distributions for nonlinear stochastic  dynamical systems.  The recursion  of the filtering  can be characterized  by the composition of Perron-Frobenius operator and likelihood operator. This gives a significant connection between the Perron-Frobenius operator and Bayesian filtering.    We numerically fulfil  the recursion  through approximating the Perron-Frobenius operator by Ulam's method.  In this way, the posterior measure is   represented  by a convex combination of the indicator functions in Ulam's method.
To get a low rank approximation for the Perron-Frobenius operator filter, we take  a spectral decomposition for the posterior measure by using  the eigenfunctions of the discretized Perron-Frobenius operator.  A convergence analysis is carried out and shows that the Perron-Frobenius operator filter achieves a higher convergence rate than the particle filter, which uses Dirac measures for the posterior. The proposed method is explored  for the data assimilation of  the stochastic dynamical systems. A few numerical examples are presented to illustrate the advantage  of the Perron-Frobenius operator filter over particle filter and extend Kalman filter.

\smallskip
{\bf keywords:} Perron-Frobenius operator, Bayesian   filtering, stochastic dynamical systems, particle filter

\section{Introduction}

In recent years, the operator-based approach has been  extensively exploited to analyze dynamical systems. The two primary candidates of the approach  are Perron-Frobenius operator and its dual operator, Koopman operator. Many data-driven methods have been developed for numerical approximation of these operators. The two operators are motivated to approximate the dynamical  system's behavior from different perspectives. The  Koopman  operator  is used to study the evolution of observations, while Perron-Frobenius operator (PFO) characterizes the transition of  densities. Therefore, the PFO deals with the system's uncertainties in the form of probability density functions of the state. In practice, it determines an absolutely continuous probability measure preserved by a given measurable transformation on a measure space.

The Perron-Frobenius operator has been widely used to characterize the global asymptotic behavior of dynamical systems derived from many different domains such as fluid dynamics \cite{appl2}, molecular dynamics \cite{appl4}, meteorology and atmospheric sciences \cite{appl6}, and to  estimate invariant sets or metastable sets with a toolbox like in \cite{GIAO}. It is of great interest to study  the invariant density of  Perron-Frobenius operator \cite{ind1} and design  efficient numerical approaches. Then  one can apply ergodic theorems to the statistical properties  of  deterministic dynamical systems.

Since  PFO is able to transport density of a Markov process, its approximation is necessary for numerical model transition probability of the Markov process. Many different numerical methods \cite{Ulam1}, such as Ulam's method and Petrov-Galerkin method, are proposed for approximation of the Perron-Frobenius operator. As the PFO operates on infinite-dimensional spaces, it is natural to project it onto a finite-dimensional subspace spanned by suitable basis functions. The projection is usually accomplished by Galerkin methods with weak approximation. It was originally  proposed by Ulam \cite{Ulam}, who suggested that one can study the discrete Perron-Frobenius operator on the finite-dimensional subspace $L^1$ of indicator functions according to a finite partition of the region. Convergence analysis of Ulam's method is discussed in many literatures \cite{conv1,conv3}.


The classical filtering problems in stochastic processes are investigated  in \cite{ExKF,sds3}. In the paper, the models of filtering problems are considered with discrete-time and continuous-time stochastic processes defined by the solutions of SDEs, which can model a majority of stochastic dynamical systems in the real world. The filtering methods have been widely used for  geophysical applications, such as oceanography \cite{ocea}, oil recovery \cite{oil}, atmospheric science, and weather forecast \cite{atmos}. Remarkably, as one of the filtering methods,  Kalman filter \cite{Kalman} has been well-known  for low-dimensional engineering applications in linear Gaussian models, and it has been also developed and utilized in many other fields \cite{KF1,KF2, KF-DMD}.  For nonlinear problems, the classical filters, such as 3DVAR \cite{3dvar}, Extended Kalman filter \cite{ExKF} and Ensemble Kalman filter \cite{EnKF}, usually invoke a Gaussian ansatz. They are often used in the scenarios with small noisy observation and high dimensional spaces. However, these extensions rely on the invocations of Gaussian assumption. As a sequential Monte Carlo method,    particle filter \cite{pf1} is able to work well for the  nonlinear and non-Gaussian filtering problems. It can be proved to estimate true posterior filtering problems in the limit of  large number of particles.

Although the particle filter (PF) can treat the nonlinear and non-Gaussian filtering problems, it has some limitations in practice. First of all, particle filter handles well in low-dimensional systems, but it may occur degeneracy \cite{degeneracy} when the systems have very large scale. It means that the maximum of the weights associated with the sample ensemble converges to one as the system dimension tends to infinity.  To avoid degeneracy, it requires a great number of particles that scales exponentially  with the system dimension. This is a manifestation of the curse of dimensionality. Resampling, adding jitter and localisation are introduced to circumvent this curse of dimensionality and get the accurate estimation of high-dimensional probability density functions (PDFs). The particle filter also does not  perform well in geophysical applications of data assimilation, because the data in these application have  strongly constraints on particle location  \cite{DA-ma}. This impacts on the filtering performance. Besides, the prior knowledge of the transition probability density functions in particle filter is necessary to be known, and the efficient sampling methods such as acceptance-rejection method and Markov chain Monte Carlo, need to be used for  complicated  density functions. The sampling is particularly a challenge in high dimensional  spaces.      To overcome these difficulties,  we propose a Perron-Frobenius operator filter (PFOF), which does not use particles and any sampling methods. The information of prior probability density is  not required  in the method, which needs data information instead. The method works  well in nonlinear and non-Gaussian models.

In this paper, we propose  PFOF  to treat  nonlinear stochastic  filtering problems. The method transfer filtering distribution with the Perron-Frobenius operator. For filtering problems, the update of filtering distribution involves   two steps: predication and analysis. In prediction, the density is transported with a transition density function given by a Markov chain of the underlying system. In analysis, the density is corrected  by Bayes' rule under the likelihood function given by observations. Thus, the update of filtering distribution can be expressed as a composition of  PFO and likelihood functions. In the simulation process,  Ulam' method is used to discretize the PFO and project it onto a finite-dimensional space of indicator functions. Hence the filtering density is also projected onto the subspace and is represented by a linear convex combination of the indicator functions. The recursion  of filtering distribution is then expressed by a linear map of weights vectors associated with the basis functions via the PFO and likelihood function. For the high dimensional problems, we propose a low-rank PFOF (lr-PFOF) using a spectral decomposition. To this end,  we first use the eigenfunctions of the discretized PFO to represent the spectral decomposition of the density functions. Then we make a truncation of the decomposition and use the  eigenfunctions corresponding to the first few dominant eigenvalues. This can improve the online assimilation efficiency. The idea of PFO  is extended to the continuous-time filtering problems. In these problems, Zakai equation  characterizes  the transition of filtering density. We utilize the approximation of the  Perron-Frobenius operator to compute the Zaikai equation  and obtain   the posterior density functions  of the continuous-time filtering problems.

We compare the proposed method with the particle filter.  For PFOF, we give an error estimate in the total-variance distance between the approximated posterior  measure and the truth posterior. The estimate implies  that PFOF achieves a convergence rate $O({1\over N})$, which is  faster than particle filters with the same number $N$  of basis functions. Our numerical results  show that PFOF  also renders  better accuracy  than that of extended Kalman filter.

The rest  of the paper is organized as follows. In Section \ref{sect2}, we express  the   Bayesian filter in terms of the Perron-Frobenius operator. In Section \ref{sect3}, we present the recursion  of the filtering empirical measure with an approximated Perron-Frobenius operator. Then we derive PFOF as well as lr-PFOF, and analyze an error estimate subsequently. PFOF is also extended to the continuous-time filtering problems. A comprehensive comparison  with particle filter is give in Section \ref{sect4}. A few numerical results of stochastic filtering  problems are given in Section \ref{sect5}. Section \ref{sect6} concludes the paper in a summary.

\section{Preliminaries}
\label{sect2}

We give a background review on Perron-Frobenius operator \cite{PFO} (PFO) and Bayesian filter in this section. The Perron-Frobenius operator  transports the distributions over state space and describes the  stochastic behavior of the dynamical systems. The framework of Bayesian filter is introduced and summarized as a recursive formula with PFO.

\subsection{Perron-Frobenius operator}
\label{sect2-1}
Let $X$ be a metric space, $\mathcal{B}$ the Borel-$\sigma$-algebra on $X$, and $\Phi:X\rightarrow X$ a nonsingular transformation. Let $\mathcal{M}$ denote the space of all finite measures on (X, $\mathcal{B}$) and $\mu$ is a finite measure. The phase space is defined on   a measure space (X, $\mathcal{B}$, $\mu$). The Perron-Frobenius operator $\mathcal{P}:\mathcal{M}\rightarrow\mathcal{M}$ is a linear and  infinite-dimensional operator defined by
\begin{equation}
\label{measure-PFO}
\mathcal{P}\mu(A)=\mu(\Phi^{-1}(A)), \quad \forall A\in \mathcal{B}.
\end{equation}
The PFO is a linear, positive and  non-expansive operator, and  hence a Markov operator. We can also track the action on distributions in the  function space. In the paper, we denote $L^1(X):=L^1(X,\mathcal{B},\mu)$.   Let  $f\in L^1(X)$ be  the probability density function (PDF) of a $X$-valued random variable $x$. Since $\Phi$ is a nonsingular with respect to $\mu$, there is a $g\in L^1(X)$ satisfying $\int_{\Phi^{-1}(A)}fd\mu=\int_{A}gd\mu$ for all $A\in \mathcal{B}$. Then g is the function characterizing the distribution of $\Phi(x)$.  The mapping $f\mapsto g$, defined uniquely by a linear operator $\mathcal{P}:L^1(X)\rightarrow L^1(X)$ :
\begin{equation}
\label{density-PFO}
\int_{A}\mathcal{P}f \,d\mu=\int_{\Phi^{-1}(A)}f \,d\mu, \quad \forall A\in \mathcal{B}.
\end{equation}
The  operator $\mathcal{P}$ is called the Perron-Frobenius operator. With the definition (\ref{measure-PFO}) and (\ref{density-PFO}), we  make the connection between  probability density function and the  measure associated with the PFO. When $f$ is a probability density function with respect to an absolutely continuous probability measure $\mu\in \mathcal{M}(X)$, $g$ is another PDF with respect to the absolutely continuous probability measure $\mu\circ \Phi^{-1}$. In addition, the measure $\mu\in\mathcal{M}(X)$ is an invariant measure of $\mathcal{P}$ when  $\mathcal{P}\mu=\mu$ holds.

 Let $\Psi:\mathbb{R}_{+}\times X \rightarrow X$  be a nonsingular flow map for a deterministic continuous-time dynamical system. Then  $\Psi_{\tau}:X\rightarrow X$ is nonsingular for each $\tau\in\mathbb{R}^{+}$. The transfer operator $\mathcal{P}_{\tau}:L^1(X)\rightarrow L^1(X)$ is time-dependent and has an analogous definition to (\ref{density-PFO}), such that
 \[
 \int_{A}\mathcal{P}_{\tau}f \,d\mu=\int_{\Psi_{\tau}^{-1}(A)}f \,d\mu.
  \]
The  $\{\mathcal{P}_{\tau}\}_{\tau\geq0}$ forms a  semigroup of  the Perron-Frobenius operators. We note  that  $\{\mathcal{P}_{\tau}\}_{\tau\geq0}$ has an infinitesimal generator $\mathscr{A}_{PF}$ by Hille-Yosida Theorem.

Let us consider the  Perron-Frobenius operator in stochastic dynamic systems and  the infinitesimal generator of PFO associated to the stochastic solution semiflow induced by a stochastic dynamical equation (SDE). Let $b:X\rightarrow X$ and $\sigma:X\rightarrow X$ be smooth time-invariant functions. Suppose that a stochastic process $x_t$ is the solution to the time-homogeneous stochastic differential equation:
\begin{equation}
\label{SDE-dynamic}
dx_t=b(x_t)dt+\sigma(x_t) dW_t, \quad x(t_0)\thicksim \rho_0,
\end{equation}
where $W_t$  is a standard Brownian motion. In this case, the distribution of the stochastic process  $x_t$ can be described by a semigroup of Perron-Frobenius operators $\{\mathcal{P}_{\tau}\}_{\tau> 0}$ on $L^{1}(X)$. The generator of $\{\mathcal{P}_{\tau}\}_{\tau> 0}$ is a second-order differential operator on $X$. The PDE defined by the generator describes the evolution of the probability density  of $x_t$.

Suppose that $\Phi$ is the mapping of the stochastic dynamical system (\ref{SDE-dynamic}) and $\Phi(x)$ is an $X$-valued random variable over the probability space (X, $\mathcal{B}$, $\mu$). Given a stochastic transition function $p_{\tau}:X\times\mathcal{B}\rightarrow[0,1]$ induced by $\Phi$, we consider probability measure $\mu$ translated with a linear operator defined in terms of the transition function $p_{\tau}(x,\cdot)$.  The stochastic PFO \cite{SPFO} $\mathcal{P}_{\tau}:\mathcal{M}\rightarrow\mathcal{M}$ is defined by
\begin{equation}
\label{measure-trans}
\mathcal{P}_{\tau}\mu(A)=\int_{X}p_{\tau}(x,A)\,d\mu(x), \quad \forall A\in \mathcal{B}.
\end{equation}
If $p_{\tau}(x,\cdot)$ is absolutely continuous to $\mu$ for all $x\in X$, there exists a nonnegative transition density function $q_{\tau}:X\times X\rightarrow \mathbb{R}$ with $q_{\tau}(x,\cdot)\in L^1(X)$ and
\begin{equation*}
\label{prob-trans}
P(x_{t+\tau}\in A|x_t=x)=\int_Aq_{\tau}(x,y)d\mu(y),\quad \forall A\in\mathcal{B}.
\end{equation*}
The transition density function is the infinite-dimensional counterpart of the transition matrix for  a Markov chain. Now we define the stochastic PFO associated with transition density. If $f\in L^1(X)$ is a probability density function, the Perron-Frobenius semigroup of operators $\mathcal{P}_{\tau}:L^1(X)\rightarrow L^1(X), \tau>0$, is defined by
\begin{equation*}
\label{density-trans}
\mathcal{P}_{\tau}f(y)=\int_{X}q_{\tau}(x,y)f(x)\,d\mu(x).
\end{equation*}
The PFO $\mathcal{P}_{\tau}$ defined here translates the probability density function of $x_t$ with time. Let $\rho$ be a probability density. The infinitesimal generator $\mathscr{A}_{PF}$ of $\mathcal{P}_{\tau}$ is given by
\begin{equation*}
\label{PFoperator}
\mathscr{A}_{PF}\rho=-\nabla\cdot(b\rho)+\frac{1}{2}\nabla\cdot \nabla \cdot (\sigma \sigma^T  \rho).
\end{equation*}
We assume that $\widetilde{\rho}:[0,\infty)\times X\rightarrow [0,\infty)$ is the probability density function of the solution $x_t$ in (\ref{SDE-dynamic}) and $\rho_0$ is the density function of the initial condition $x_0$. Then  $\widetilde{\rho}$ solves the Fokker-Planck equation,
\begin{equation*}
\label{F-P}
   \left\{
 \begin{aligned}
\frac{\partial \widetilde{\rho}}{\partial t}&=-\nabla\cdot(b\widetilde{\rho})+\frac{1}{2}\nabla\cdot\nabla \cdot (\sigma \sigma^T \widetilde{\rho}),\quad (t,x)\in  (0,\infty)\times X,\\
\widetilde {\rho}(0,x)&= \rho_0(x).
 \end{aligned}
 \right.
\end{equation*}
If the phase space $X$ is compact and $b\in C^3(X,X)$, the equation has a unique solution, which is given by  \[
\widetilde {\rho}(t,x)=\mathcal{P}_{t}\rho_0(x).
\]

\subsection{Bayesian filter}

In this section, we present the framework of Bayesian filter in discrete time from the perspective of  Bayes' rule. In filtering problems, a state model  and an observation model are combined  to estimate the posterior distribution, which is a conditional distribution  of the state given by observation. Let us consider the dynamical model governed by the flow $\Psi\in C(\mathbb{R}^n,\mathbb{R}^n)$ with noisy observations $y=\{y_j\}_{j\in\mathbb{Z}^{+}}$ depending on the  function $h(x):\mathbb{R}^n\rightarrow\mathbb{R}^p$:
\begin{equation}\label{dynamic}
		\left\{
\begin{aligned}
x_{j+1}&=\Psi(x_j)+\xi_j,\; j\in\mathbb{N},\;x_0\thicksim \rho_0,\\
y_{j+1}&=h(x_{j+1})+\eta_{j+1},\;j\in\mathbb{N},
\end{aligned}
		\right.
	\end{equation}
where $\xi:=\{\xi_j\}_{j\in\mathbb{N}}$ is an i.i.d. sequence with $\xi_j\thicksim N(0,\Sigma)$ and $\eta:=\{\eta_j\}_{j\in\mathbb{Z}^{+}}$ is an i.i.d. sequence with $\eta_j\thicksim N(0,R)$. Let $Y_j=\{y_l\}_{l=1}^j$ denote the data up to time $t_j$. The filtering problem aims to  determine the posterior PDF $p(x_j|Y_j)$ of the random variable $x_j|Y_j$ and the sequential updating of the PDF as the data  increases. The Bayesian filtering involves two steps: prediction and analysis. It  provides a derivation of $p(x_{j+1}|Y_{j+1})$ from  $p(x_{j}|Y_{j})$. The prediction is concerned with  the map $p(x_j|Y_j)\mapsto p(x_{j+1}|Y_j)$ and the analysis derives the map   $p(x_{j+1}|Y_j)\mapsto p(x_{j+1}|Y_{j+1})$ by Bayes's formula.

\textbf{Prediction}
\begin{equation}
\label{filter-pre}
\begin{aligned}
p(x_{j+1}|Y_j)&=\int_{\mathbb{R}^n}p(x_{j+1}|Y_j,x_j)p(x_{j}|Y_j)dx_j \\
&=\int_{\mathbb{R}^n}p(x_{j+1}|x_j)p(x_{j}|Y_j)dx_j.
\end{aligned}
\end{equation}
Note that $p(x_{j+1}|Y_j,x_j)=p(x_{j+1}|x_j)$, because $Y_j$ provides indirect information about determining $x_{j+1}$. Since $p(x_{j+1}|x_j)$ is specified by the underlying model (\ref{dynamic}) and

\begin{equation}
\label{filter-trans}
p(x_{j+1}|x_j)\propto {\rm exp}(-\frac{1}{2}|\Sigma^{-\frac{1}{2}}(x_{j+1}-\Psi(x_j))|^{2}),
\end{equation}
the prediction provides the map from $p(x_{j}|Y_j)$ to $p(x_{j+1}|Y_j)$. Let $\widehat{\mu}_j$be the prior probability measure corresponding to the density $p(x_{j}|Y_{j-1})$ and $\mu_j$  be the posterior probability measure on corresponding to the density $p(x_{j}|Y_j)$. The stochastic process $\{x_j,j\in\mathbb{N}\}$ of (\ref{dynamic}) is a Markov chain with the transition kernel $p(\cdot,\cdot)$ determined by $p(x_{j},x_{j+1})=p(x_{j+1}|x_{j})$. Then we can rewrite (\ref{filter-pre}) as
\begin{equation}
\label{filter-pre1}
\widehat{\mu}_{j+1}(\cdot)=(\mathcal{P}\mu_j)(\cdot):=\int_{\mathbb{R}^n}p(x_j,\cdot)\mu_j(dx_j)=\int_{\mathbb{R}^n}p(x_j,\cdot)d\mu(x_j)\footnote{Refer to \cite{measure}, if the function $f\in L^1(X)$ on a measure space (X, $\mathcal{B}$, $\mu$) is said to be $\mu$ integrable, we have $\int fd\mu=\int f(x)\mu(dx)=\int f(x)d\mu(x)$.}.
\end{equation}
In particular, the operator  $\mathcal{P}$ coincides with the  Perron-Frobenius operator defined in (\ref{measure-trans}).

\textbf{Analysis}
\begin{equation}
\label{filter-ana}
\begin{aligned}
p(x_{j+1}|Y_{j+1})&=p(x_{j+1}|Y_{j},y_{j+1})\\
&=\frac{p(y_{j+1}|x_{j+1},Y_{j})p(x_{j+1}|Y_{j})}{p(y_{j+1}|Y_{j})}\\
&=\frac{p(y_{j+1}|x_{j+1})p(x_{j+1}|Y_{j})}{p(y_{j+1}|Y_{j})}.
\end{aligned}
\end{equation}
Note that $p(y_{j+1}|x_{j+1},Y_{j})=p(y_{j+1}|x_{j+1})$ and Bayes's formula is used in the second equality. The likelihood function $p(y_{j+1}|x_{j+1})$ is determined by the observation model: $p(y_{j+1}|x_{j+1})\propto{\rm exp}(-\frac{1}{2}|R^{-\frac{1}{2}}(y_{j+1}-h(x_{j+1}))|^2)$. Let
\begin{equation}
\label{likelihood}
g_j(x_{j+1}):={\rm exp}(-\frac{1}{2}|R^{-\frac{1}{2}}(y_{j+1}-h(x_{j+1}))|^2).
\end{equation}
The analysis provides a map from $p(x_{j+1}|Y_{j})$ to $p(x_{j+1}|Y_{j+1})$, so we can represent the update of
the measure $\mu_{j+1}(\cdot)$ by
\begin{equation}
\label{filter-ana1}
\mu_{j+1}(\cdot)=(L_j\widehat{\mu}_{j+1})(\cdot):=\frac{g_j(x_{j+1})\widehat{\mu}_{j+1}(\cdot)}
{\int_{\mathbb{R}^n}g_j(x_{j+1})\widehat{\mu}_{j+1}(\cdot)},
\end{equation}
where the likelihood operator $L_j$ is defined by
\begin{equation}
\label{L-O}
(L_j\mu)(dx)=\frac{g_j(x)\mu(dx)}{\int_{\mathbb{R}^n}g_j(x)\mu(dx)}.
\end{equation}
In general, the prediction and analysis provide the mapping from $\mu_j$ to $\mu_{j+1}$. The prediction  maps $\mu_j$ to $\widehat{\mu}_{j+1}$ through the Perron-Frobenius operator $\mathcal{P}$, while the analysis maps $\widehat{\mu}_{j+1}$ to $\mu_{j+1}$ through the likelihood  operator $L_j$. Then we represent the $\mu_{j+1}$  using formulas  (\ref{filter-pre1}) and (\ref{filter-ana1}), and summarize Bayesian filtering as
\begin{equation}
\label{filter-op}
\mu_{j+1}=L_j\mathcal{P}\mu_{j}.
\end{equation}
The $\mu_{0}$ is assumed to be a known initial probability measure. We note that $\mathcal{P}$ does not depend on $j$, because the prediction step is governed by the same Markov process at each $j$. However, $L_j$ depends on $j$ because  the different observations are used to compute the likelihood at each $j$. In this way, the evolution of $\mu_j$ processes through a linear operator $\mathcal{P}$ and a nonlinear operator $L_j$. The approximation of $\mu_j$ can be achieved by the numerical iteration of (\ref{filter-op}).


\section{Bayesian filter in terms of Perron-Frobenius operator}
\label{sect3}
It is noted  that the Perron-Frobenius operator translates a probability density function with time according to the flow of the dynamics. We extend the idea to  filtering problems to represent the transition of the posterior probability density function, i.e., the filtering distribution. Therefore, we propose a filtering method: Perron-Frobenius operator filter (PFOF). In the proposed method, the density function is projected onto an approximation subspace spanned by indicator functions. The fluctuation of the density function, which is approximated by weights vector,  is  transferred by PFO and likelihood operator. Moreover, we present a low-rank Perron-Frobenius operator filter (lr-PFOF), which is a modified version of the PFOF.

\subsection{Perron-Frobenius operator filter}
The iteration (\ref{filter-op}) is helpful to design a  filter method. According to definition (\ref{filter-pre1}), the operator  $\mathcal{P}$ in the iteration is Perron-Frobenius operator corresponding to the flow $\Psi$  of the model (\ref{dynamic}). Based on the idea, we propose a Perron-Frobenius operator filter, which utilizes Ulam's method \cite{Ulam}  to approximate operator $\mathcal{P}$ in the iteration. In PFOF, we simply use $\mathcal{P}$ for  $\mathcal{P}_{\tau}$  because the discrete time steps of the state model keep the same. In this manner, the iteration of filtering distribution of PFOF becomes

\begin{equation}
\label{PFOF-op}
\mu^N_{j+1}=L_j\mathcal{P}^N\mu^N_j, \quad \mu^N_0=\mu_0,
\end{equation}
where $\mathcal{P}^N$ calculated by the Ulam's method is an approximation of $\mathcal{P}$. Ulam's method  is a Galerkin projection  method to discretize the Perron-Frobenius operator. We first give a discretisation of the phase space. Let $B=\{\mathbb{B}_1,\cdots,\mathbb{B}_N\}\subset\mathcal{B}$ be a finite number of measure boxes and a disjoint partition of phase space $X$. The indicator function is a piecewise constant function and is defined  by
\begin{equation}
\label{indi}
\mathds{1}_{\mathbb{B}_i}(x)=\left\{
\begin{aligned}
&1,&  \rm if\;x\in \mathbb{B}_i,\\
&0,&  \rm otherwise.&
\end{aligned}
\right.
\end{equation}
Ulam proposed to use the space of  a family of indicator functions $\{\mathds{1}_{\mathbb{B}_1},\cdots,\mathds{1}_{\mathbb{B}_N}\}$ as the approximation space for the  PFO. We define the projection space $V_N:={\rm span}\{\mathds{1}_{\mathbb{B}_1},\cdots,\mathds{1}_{\mathbb{B}_N}\}$. The $V_N\in L^1(X)$ is regarded as an approximation subspace of $L^1(X)$. For each $\rho\geq 0$ in $L^1(X)$, we define  the operator $\pi_N:L^1(X)\rightarrow V_N$  by
\begin{equation}
\label{pi-N}
\pi_N \rho = \sum_{i=1}^N\omega^{(i)}\mathds{1}_{\mathbb{B}_i}, \quad {\rm where} \quad \omega^{(i)}:=\frac{\int_{\mathbb{B}_i}\rho\,d\mu}{\mu(\mathbb{B}_i)}.
\end{equation}
Then $\pi_N$ is the projection onto $V_N$. Due to the projection, we define the discretized PFO $\mathcal{P}^N$ as
\begin{equation*}
\mathcal{P}^N=\pi_N\circ \mathcal{P}.
\end{equation*}
We can represent the  linear map $\mathcal{P}^N|_{V^1_N}: V_N^1\rightarrow V_N^1, \; \text{where} \; V_N^1:=\big\{f\in V_N:\int|f|d\mu=1\big \}$ by a matrix  $P^N=(P^N_{ij})\in\mathbb{R}^{N\times N}$ whose  entries  $P^N_{ij}=\frac{1}{\mu(\mathbb{B}_i)}\int_{\mathbb{B}_i}\mathcal{P}\mathds{1}_{\mathbb{B}_j}d\mu$. The entries characterizes  the transition probability   from the box $\mathbb{B}_i$ to box $\mathbb{B}_j$ under the flow map  $\Psi$. We can show
\begin{equation}
\label{pij}
P^{N}_{ij}=\frac{\mu(\mathbb{B}_i\cap \Psi^{-1}(\mathbb{B}_j))}{\mu(\mathbb{B}_i)}.
\end{equation}
A Markov chain for $\Psi$ arises as the discretization $\mathcal{P}^N$  of the PFO, and the Markov chain   has transition matrix $P^{N}$. So the Ulam's method can be described either in terms of  the operator $\mathcal{P}^N$ or the matrix $P^N$. By the projection $\pi_N$, the  density $\rho$ can be expressed as a vector $\mathbf{W}=[\omega^{(1)},\cdots,\omega^{(N)}]$, where $\omega^{(i)}$ is  the weight of the basis function $\mathds{1}_{\mathbb{B}_i}$. Since  the entries $P^{N}_{ij}$ represent the transition probability from $\mathbb{B}_i$ to $\mathbb{B}_j$, they can be estimated by a Monte-Carlo method, which gives  a numerical realization of Ulam's method \cite{Ulam1}.  We randomly choose a large number of points $\{x_i^l\}_{l=1}^n$ in each $\mathbb{B}_i$ and count the number of times $\Psi(x_i^{l})$  contained in box $\mathbb{B}_j$. Then  $P^{N}_{ij}$ is calculated by
\begin{equation}
\label{pij-est}
P^{N}_{ij}\approx P^N_{n,ij}=\frac{1}{n}\sum_{l=1}^n\mathds{1}_{\mathbb{B}_j}(\Psi(x_i^l)).
\end{equation}
The Monte-Carlo method is used as an approximation to the integrals (\ref{pij}). The convergence of the Ulam's method   depends on the choice of the partition of the region and the number of points. Based on indicator basis functions, we denote the the empirical density in the PFOF with respect to the measure $\mu_j^N$ as
\begin{equation}
\label{empdf}
 \rho_j^N(x)=\sum_{i=1}^{N} \omega^{(i)}_j\mathds{1}_{\mathbb{B}_i}(x),
\end{equation}
where $\mathds{1}_{\mathbb{B}_i}(\cdot)$ is the  indicator function defined by (\ref{indi}) and $j$ represents the index of time sequence. In this way, the density $\rho_j^N$ can be represented  by the vector of the weights  $\mathbf{W}_{j} = [\omega_j^{(1)},\cdots,\omega_j^{(N)}]$. Suppose that $\mathbf{W}_{j} = [\omega_j^{(1)},\cdots,\omega_j^{(N)}]$ and $\mathbf{W}_{j+1} = [\omega_{j+1}^{(1)},\cdots,\omega_{j+1}^{(N)}]$ are separately the weights of $\pi_N \rho_j$ and $\pi_N \rho_{j+1}:=\pi_N \mathcal{P}\rho_j$. When the region is evenly divided, the evolution of density functions becomes a Markov transition equation of the weights:

\begin{equation}
\label{Markov-trans}
\mathbf{W}_{j+1}=\mathbf{W}_{j}P^{N},
\end{equation}
where $P^N$ is the matrix form of discretized PFO. We consider the projection of  the $\mathcal{P}\rho_j$, i.e.,
\begin{equation}
\label{pi-PT}
\pi_N\mathcal{P}\rho_j= \sum_{i=1}^{N} \omega^{(i)}_{j+1}\mathds{1}_{\mathbb{B}_i}.
\end{equation}
In addition, note that
\begin{equation*}
\pi_N\mathcal{P}\rho_j = \pi_N\mathcal{P}\sum_{i=1}^{N} \omega^{(i)}_j\mathds{1}_{\mathbb{B}_i}
= \sum_{i=1}^{N} \omega^{(i)}_j \pi_N(\mathcal{P}\mathds{1}_{\mathbb{B}_i}).
\end{equation*}
We denote $\pi_N(\mathcal{P}\mathds{1}_{\mathbb{B}_i})=\sum_{k=1}^{N} c_{ik}\mathds{1}_{\mathbb{B}_k}$, where
\begin{equation*}
\begin{aligned}
c_{ik}=\frac{\int_X\mathcal{P}(\mathds{1}_{\mathbb{B}_i})\mathds{1}_{\mathbb{B}_k}dx}{\mu(\mathbb{B}_k)}= \frac{\int_{\mathbb{B}_k}\mathcal{P}(\mathds{1}_{\mathbb{B}_i})dx}{\mu(\mathbb{B}_k)}\\
= \frac{\int_{\Psi^{-1}(\mathbb{B}_k)}\mathds{1}_{\mathbb{B}_i}dx}{\mu(\mathbb{B}_k)}= \frac{\mu(\mathbb{B}_i\cap \Psi^{-1}(\mathbb{B}_k))}{\mu(\mathbb{B}_k)}.
\end{aligned}
\end{equation*}
Then we have
\begin{equation*}
\begin{aligned}
\pi_N\mathcal{P}\rho_j&=\sum_{i=1}^{N}\omega^{(i)}_j\sum_{k=1}^{N}c_{ik}\mathds{1}_{\mathbb{B}_k}=\sum_{k=1}^{N}\sum_{i=1}^{N}\omega^{(i)}_jc_{ik}\mathds{1}_{\mathbb{B}_k}\\
&=\sum_{k=1}^{N}\sum_{i=1}^{N}\frac{\mu(\mathbb{B}_i)}{\mu(\mathbb{B}_k)}\omega^{(i)}_jP^N_{ik}\mathds{1}_{\mathbb{B}_k}.
\end{aligned}
\end{equation*}
Comparing to (\ref{pi-PT}), we get

\begin{equation}
\label{weight-trans}
\omega^{(k)}_{j+1}=\sum_{i=1}^{N}\frac{\mu(\mathbb{B}_i)}{\mu(\mathbb{B}_k)}\omega^{(i)}_jP^N_{ik}.
\end{equation}
Thus, if we give a uniform partition of the $X$, i.e., $\mu(\mathbb{B}_i)=\mu(\mathbb{B}_j),\;\forall i,j\in N$, then we get the result (\ref{Markov-trans}). With the expression, we design the following prediction step and analysis step to approximate the posterior distribution $p(x_{j+1}|Y_{j+1})$.

\textbf{Prediction} In this step, we give a set of boxes $\{\mathbb{B}_1,\cdots,\mathbb{B}_N\}\subset\mathcal{B}$, which is a uniform partition  of $X$, and denote the  mass point of each box as $x^{(i)}, i=1,\cdots,N$. Define $\widehat{\mathbf{W}}_{j} = [\widehat{\omega}_j^{(1)},\cdots,\widehat{\omega}_j^{(N)}]$  as the prior weight vector and $\mathbf{W}_{j} = [\omega_j^{(1)},\cdots,\omega_j^{(N)}]$ as the posterior weight vector. In  equation (\ref{filter-pre1}), we note that the prior density $p(x_{j+1}|Y_j)$ is computed under the linear operator $\mathcal{P}$.
To discretize  the formula $\widehat{\mu}_{j+1}=\mathcal{P}\mu_j$, we build a map between the weights of the density function,

\begin{equation*}
\widehat{\mathbf{W}}_{j+1}=\mathbf{W}_jP^N.
\end{equation*}
 The formula contains the prior information of the underlying system (\ref{dynamic}) because the PFO in the formula is defined by the transition kernel $p$ of the system. With the basis functions $\mathds{1}_{\mathbb{B}_i}(\cdot)$, the empirical prior measure is given by
\begin{equation*}
\widehat{\mu}_{j+1}^N = \sum_{i=1}^{N}\widehat{\omega}^{(i)}_{j+1}\mathds{1}_{\mathbb{B}_i}(dx).
\end{equation*}

\textbf{Analysis} In this step, we derive the posterior measure $\mu_{j+1}^N$. To achieve this, we apply Bayes's formula (\ref{filter-ana}) on weights and update the weights by
\begin{equation}
\label{reweight}
\omega^{(i)}_{j+1}=\widetilde{\omega}^{(i)}_{j+1}/(\sum_{n=1}^N\widetilde{\omega}^{(n)}_{j+1}),
\quad \widetilde{\omega}^{(i)}_{j+1}=g_j(x^{(i)})\widehat{\omega}^{(i)}_{j+1},
\end{equation}
where $g_j(x)$ given by (\ref{likelihood}) denotes the likelihood function as before. Then the $\mu_{j+1}^N$ approximated by the indicator measure is given by
\begin{equation*}
\mu_{j+1}^N = \sum_{i=1}^{N}\omega^{(i)}_{j+1}\mathds{1}_{\mathbb{B}_i}(dx).
\end{equation*}
Note that we choose the mass point $x^{(i)}$ of each box $\mathbb{B}_i$ to calculate  $g_j(x)$, i.e., the likelihood function. It is a reasonable choice to approximate the likelihood function of the points  in the $\mathbb{B}_i$. In both prediction step and analysis step, they only evolve weights $\{\omega^{(i)}_j\}_{i=1}^N$ into $\{\omega^{(i)}_{j+1}\}_{i=1}^N$ via $\{\widehat{\omega}^{(i)}_{j+1}\}_{i=1}^N$, and provide a transform  from $\mu_j^N$ to $\mu_{j+1}^N$. The complete procedure is summarized in Algorithm \ref{alg:2}, named Perron-Frobenius operator filter. The algorithm consists of two phases: offline phase to compute $P^N$ by Ulam's method, and online phase to update the approximation of the posterior measure. Besides, the standard   Ulam's method becomes inefficient in high-dimensional dynamical systems due to the curse  of dimensionality. For this case, we may use the sparse Ulam method \cite{sUlam} instead. It constructs an  optimal approximation subspace and costs less computational effort than the standard Ulam's method when a certain accuracy is imposed.

\begin{algorithm}[htb]
	\caption{Perron-Frobenius operator filter}
	\label{alg:1}
	\begin{algorithmic}
     	\STATE \textbf{Offline:}
        \STATE Compute $P^N$ by Ulam's method
     \end{algorithmic}
     \begin{algorithmic}
        \STATE \textbf{Online:}
     \end{algorithmic}
     \begin{algorithmic}[1]
		\STATE Set $j=0$ and $\mu_0^N(dx_0)=\mu_0(dx_0)$, compute $\omega^{(i)}_0=\frac{\int_{\mathbb{B}_i}\mu_0dx_0}{\mu(\mathbb{B}_i)}$
		\STATE Let $\mathbf{W}_j=[\omega^{(1)}_j,\cdots,\omega^{(N)}_j]$, compute $\widehat{\mathbf{W}}_{j+1}=\mathbf{W}_jP^N$
        \STATE Define $\widehat{\mu}_{j+1}^N = \sum_{i=1}^{N}\widehat{\omega}^{(i)}_{j+1}\mathds{1}_{\mathbb{B}_i}(x)$
        \STATE Denote $\omega^{(i)}_{j+1}$ by $(\ref{reweight})$, define $\mu_{j+1}^N = \sum_{i=1}^{N}\omega^{(i)}_{j+1}\mathds{1}_{\mathbb{B}_i}(x)$
        \STATE j+1$\rightarrow$ j
        \STATE Go to step 2
	\end{algorithmic}
\end{algorithm}


\subsection{Analysis of error estimate}

We analyze   the error estimate of the Perron-Frobenius operator filter in this section to explore the factors, which  determine  convergence of the algorithm. Define a total-variation distance $d(\cdot,\cdot)$ between two
 probability measures $\mu$ and $\nu$ as follows:
\begin{equation*}
 d(\mu,\nu)=\frac{1}{2}{\rm sup}_{|f|_{\infty}\leq1}|\mathbb{E}^{\mu}(f)-\mathbb{E}^{\nu}(f)|,
\end{equation*}
where $\mathbb{E}^{\mu}(f)=\int_{X}f(x)\mu(dx)$ for $f\in L^1(X)$ and $|f|_{\infty}={\rm sup}_{x}|f(x)|$. The
distance $d(\cdot,\cdot)$ can also be characterized by the $L^1$ norm of the difference between the two PDFs $\rho$ and $\rho'$, which  correspond  to the measure $\mu$ and $\nu$, respectively, i.e.,
\begin{equation}
\label{TV-func}
 d(\mu,\nu)=\frac{1}{2}\int_X|\rho(x)-\rho'(x)|dx.
\end{equation}
The distance induces a metric. To estimate the error, we recall the iteration (\ref{PFOF-op}) and see that the approximation error of the probability comes from the operator $\mathcal{P}^N$. To do this, we need  the  following lemmas.

\begin{lem}{\rm (Theorem 4.8 in \cite{DA-ma})}
\label{lem1}
Suppose that $\mathcal{P}$ is the Perron-Frobenius operator defined in  (\ref{filter-pre1}). Let $\mu$ and $ \nu$ be two arbitrary probability measures. Then
\begin{equation*}
d(\mathcal{P}\mu,\mathcal{P}\nu)\leq d(\mu,\nu).
\end{equation*}
\end{lem}

\begin{lem}{\rm (Lemma 4.9 in \cite{DA-ma})}
\label{lem2}
Let $g_j$ be the likelihood function defined by (\ref{likelihood}) and the operator $L_j$  defined by (\ref{L-O}). Assume that there exists $\kappa\in(0,1]$ such that for all $x\in X$ and $j\in \mathbb{N}$,
\begin{equation}
\label{g-k}
\kappa\leq g_j(x)\leq \kappa^{-1}.
\end{equation}
Then we have
\begin{equation*}
 d(L_j\mu,L_j\nu)\leq 2\kappa^{-2}d(\mu,\nu).
\end{equation*}
\end{lem}

\begin{lem}{\rm (Theorem 2.4.1 in  \cite{pi-N})}
\label{lem3}
Let $\mathcal{C}_a$ be discrete Lipschitz cone defined as
\[
\mathcal{C}_a=\{\phi:\frac{\phi(x)}{\phi(y)}\leq e^{a|x-y|}, \forall x,y\in \mathcal{R}\}.
\]
For each $N>0$, the $\pi_N$ given by (\ref{pi-N}) denotes the projection of $L^1$ onto $V_N$. Then for any function $f\in \mathcal{C}_a$,
\[
\|f-\pi_N f\|_{L^1}\leq(e^{a/N}-1)\|f\|_{L^1}.
\]
\end{lem}
By the above three lemmas, we analyze  total-variance distance between the approximate measure $\mu_J^N$ and the true measure $\mu_J$ and give the following theorem.
\begin{thm}
\label{Thm1-PFOF}
If $g_j(x)$ satisfies the condition (\ref{g-k}) and the  probability density $\rho_j^N$ of the measure $\mu_j^N$ satisfies $\mathcal{P}\rho_j^N\in\mathcal{C}_a,\;\forall j\in\mathbb{N}$, then
\begin{equation*}
d(\mu_J^N,\mu_J)\leq \sum_{j=1}^{J}(2\kappa^{-2})^j\frac{e^a-1}{2N}.
\end{equation*}
\end{thm}
\begin{proof}
From the formula (\ref{filter-op}) and (\ref{PFOF-op}), we apply the triangle inequality to the distance $d(\mu_{j+1}^N,\mu_{j+1})$ and get
\begin{equation*}
\begin{aligned}
d(\mu_{j+1}^N,\mu_{j+1}) &= d(L_j\mathcal{P}^N\mu_{j}^N,L_j\mathcal{P}\mu_{j}) \\
&\leq d(L_j\mathcal{P}^N\mu_{j}^N,L_j\mathcal{P}\mu_{j}^N)+d(L_j\mathcal{P}\mu_{j}^N,L_j\mathcal{P}\mu_{j}).
\end{aligned}
\end{equation*}
According to Lemma \ref{lem2} and Lemma \ref{lem1}, it follows that
\begin{equation}
\label{u-un}
\begin{aligned}
d(\mu_{j+1}^N,\mu_{j+1}) &\leq 2k^{-2}\big[d(\mathcal{P}^N\mu_{j}^N,\mathcal{P}\mu_{j}^N)+d(\mathcal{P}\mu_{j}^N,\mathcal{P}\mu_{j})\big] \\
&\leq 2k^{-2}\big[d(\mathcal{P}^N\mu_{j}^N,\mathcal{P}\mu_{j}^N)+d(\mu_{j}^N,\mu_{j})\big],
\end{aligned}
\end{equation}
Let us consider $d(\mathcal{P}^N\mu_{j}^N,\mathcal{P}\mu_{j}^N)$. Suppose that $\rho'_{j+1}$  is density function associated with the measure $\mathcal{P}\mu_j^N$. Let  $\rho^N_j$ and  $\rho^N_{j+1}$ be  the density functions of $\mu^N_j$ and $\mu^N_{j+1}$, respectively. By the definition of total-variance distance in (\ref{TV-func}), we have
\begin{equation*}
\begin{aligned}
d(\mathcal{P}^N\mu_{j}^N,\mathcal{P}\mu_{j}^N)&=\frac{1}{2}\int_{X}|\rho_{j+1}^N(x)-\rho'_{j+1}(x)|dx \\
& = \frac{1}{2}\int_{X}|\mathcal{P}^N\rho_j^N(x)-\mathcal{P}\rho_j^N(x)|dx \\
& = \frac{1}{2}\int_{X}|\pi_N\circ \mathcal{P}\rho_j^N(x)-\mathcal{P}\rho_j^N(x)|dx,
\end{aligned}
\end{equation*}
where we have used the equation  $\mathcal{P}^N=\pi_N\circ \mathcal{P}$ in    the last equality. Since $\mathcal{P}\rho_j^N(x)\in\mathcal{C}_a$, we use Lemma \ref{lem3} and get
\begin{equation}
\begin{aligned}
\label{p-pn}
d(\mathcal{P}^N\mu_{j}^N,\mathcal{P}\mu_{j}^N)& \leq \frac{1}{2}(e^{a/N}-1) \\
& \leq \frac{1}{2N}(e^a-1).
\end{aligned}
\end{equation}
With the fact that $\mu_0^N=\mu_0$, we combine  (\ref{p-pn}) with  (\ref{u-un}) and repeat the iterating to complete the proof.
\end{proof}

Theorem \ref{Thm1-PFOF} estimates the  online error of the PFOF algorithm. Since the Perron-Frobenious operator is numerically approximated offline  by  the matrix form $P_n^N$ given by (\ref{pij-est}), we  will analyze the offline error generated by the approximation. Each coefficient of $P^N_n$ is  computed  by the Monte-Carlo approximation of (\ref{pij}) using a set of the sampling points $\{x_i^l\}_{l=1}^{n}$. We show that  the matrix $P^N_n$ converge to  the matrix $P^N$.

\begin{prop}
\label{prop-cov}
If the matrix $P^N_n$ is defined  by (\ref{pij-est}) and $P^N$ is defined  by (\ref{pij}), then the following convergence in distribution holds,
\begin{equation}
\label{cov-dis}
\sqrt{n}((P^N_n)_{ij}-(P^N)_{ij})\xrightarrow[n\rightarrow\infty]{\mathcal{D}}\mathcal{N}(0,\sigma_{ij}^{n,N}),
\end{equation}
where
\begin{equation}
\label{sigma}
(\sigma_{ij}^{n,N})^2=\int_X (\mathds{1}_{\Psi_{\tau}^{-1}(\mathbb{B}_j)}\cdot\mathds{1}_{\mathbb{B}_i})^2d\mu-
(\int_X \mathds{1}_{\Psi_{\tau}^{-1}(\mathbb{B}_j)}\cdot\mathds{1}_{\mathbb{B}_i}d\mu)^2,
\end{equation}
and $\mathcal{N}(0,\sigma_{ij}^{n,N})$ is the normal distribution with the mean 0 and standard deviation $\sigma_{i,j}^{n,N}$.
\end{prop}

\begin{proof}
Note that the entries  of $P^N_n$ are given by
\[P^N_{n,ij}=\frac{1}{n}\sum_{l=1}^{n}\mathds{1}_{\mathbb{B}_j}(\Psi_{\tau}(x_i^l)),\]
which is the  Monte-Carlo approximation of
\[P^N_{ij}=\frac{\int_X \mathds{1}_{\Psi_{\tau}^{-1}(\mathbb{B}_j)}\cdot\mathds{1}_{\mathbb{B}_i}d\mu}{\int_X \mathds{1}_{\mathbb{B}_i}d\mu},\]
with sampling points $x_i^l$ drawn independently and uniformly from the box $\mathbb{B}_i$. The denominator $\int_X \mathds{1}_{\mathbb{B}_i}d\mu$ normalizes the entries $P^N_{ij}$ so that $P^N$ becomes a right stochastic matrix, with each row summing to 1. The  convergence result (\ref{cov-dis}) follows directly from the convergence  of Monte-Carlo integration \cite{MC}.
\end{proof}

Proposition \ref{prop-cov} indicates that there exits a constant $C^N(\Psi_{\tau},\alpha)$ determined  by the standard deviation $\sigma_{ij}^{n,N}$ with a given confidence rate $\alpha\in[0,1)$ such that for $m$ large enough,  the following estimate  holds with probability at least $\alpha$:
\begin{equation}
\label{cov-norm}
\parallel(P^N_n)_{ij}-(P^N)_{ij}\parallel_{\infty}\leq C^N(\Psi_{\tau},\alpha)n^{-\frac{1}{2}}.
\end{equation}
The result shows that the convergence of $P^N_n$ to $P^N$ is in $\mathcal{O}(n^{-\frac{1}{2}})$.

\subsection{A low-rank Perron-Frobenius operator filter}
In PFOF, we note that the number of blocks increases exponentially with respect to  dimensions, resulting in the number of basis functions growing rapidly. Therefore, we propose  a low-rank approximation, formed by eigenfunctions of the Perron-Frobenius operator, to represent the density.  This approach can effectively reduce the number of the required basis functions.  Let $\rho$ still be a probability density function of the dynamical system governed by $\Psi$. Then it can be  written as a linear combination of the independent eigenfunctions $\varphi_i$ of $\mathcal{P}$. So
\[\rho(x,t)=\sum_{i=1}^{\infty}a_i\varphi_i(x),\quad a_i\in \mathbb{C}.\]
Suppose that $\lambda_i$ is the eigenvalue corresponding to the eigenfunction $\varphi_i$  of $\mathcal{P}$, then
\[\mathcal{P}\rho(x,t)=\sum_{i=1}^{\infty}\lambda_ia_i\varphi_i(x).\]
Actually, the eigenfunction of the discretized PFO can be determined by the following proposition.
\begin{prop}
Let $B=\{\mathbb{B}_1,\cdots,\mathbb{B}_N\}\subset\mathcal{B}$ be a uniform partition of the phase space $X$. If $\xi$ is the left eigenvector of $P^N$ corresponding to the eigenvalue $\lambda$, then $\lambda$ is also the eigenvalue of the restricted operator $\pi_N\mathcal{P}$ with the eigenfunction $\varphi\triangleq\xi^T\mathbf{U}$, where $\mathbf{U}=[\mathds{1}_{\mathbb{B}_1}(x),\cdots,\mathds{1}_{\mathbb{B}_N}(x)]^T$.

\end{prop}

\begin{proof}
Let $\varphi=\sum_i \xi^{(i)}\mathds{1}_{\mathbb{B}_i}$. From Eq. (\ref{pi-PT}) and Eq. (\ref{weight-trans}),

\[\pi_N\mathcal{P}\varphi=\sum^N_j \sum^N_i \xi^{(i)}P^N_{ij}\mathds{1}_{\mathbb{B}_j}.\]
Since $\xi P^N=\lambda P^N$, i.e., $\lambda\xi^{(j)}=\sum_i\xi^{(i)}P^N_{ij},\; \forall j\in\mathbb{N},$ we get
\[\pi_N\mathcal{P}\varphi=\sum^N_j \lambda\xi^{(j)}\mathds{1}_{\mathbb{B}_j}=\lambda\varphi.\]
Thus, $\lambda$ is also an eigenvalue of the restricted operator $\pi_N\mathcal{P}$ with eigenfunction $\varphi$.
\end{proof}
In order to obtain the spectral expansion of the density function $\rho_{j}:=\rho(x,t_j)$, we define the  matrix $\boldsymbol{\varphi}=[\varphi_1,\cdots,\varphi_N]^T$, where $\{\varphi_i\}_{i=1}^N$ are the eigenfunctions with respect to eigenvalues $\{\lambda_i\}_{i=1}^N$, with $|\lambda_1|=1\geq|\lambda_2|\geq\cdots\geq|\lambda_N|\geq0$. Let $\rho_{0}=\mathbf{W}_{0}\mathbf{U}$ and
\[
\Xi=\begin{bmatrix}
\xi_1^T \\
\xi_2^T \\
\vdots \\
\xi_N^T
\end{bmatrix}.
\]
Then the eigenfunction is denoted as $\boldsymbol{\varphi}=\Xi\mathbf{U}$  and the density function of $\rho_{1}$ is given by
\begin{equation}
\label{expansion}
\rho_{1}=\pi_N\mathcal{P}\rho_{0}=\pi_N\mathcal{P}\mathbf{W}_{0}\mathbf{U}=\pi_N\mathcal{P}\mathbf{W}_{0}\Xi^{-1}\boldsymbol{\varphi}=\Lambda \mathbf{W}_{0}\Xi^{-1}\boldsymbol{\varphi}=\sum_{i=1}^N\lambda_i\varphi_iv_i,
\end{equation}
where $\Lambda$ is a diagonal eigenvalue matrix for $\pi_N\mathcal{P}$ and $v_i$ is the column vector of the matrix $V=\mathbf{W}_{0}\Xi^{-1}$. The first $r$ major eigenvalues and their corresponding eigenfunctions are used to approximate the density function. If the formula (\ref{expansion}) is truncated by $r<N$, then the low-rank model of $\rho_{1}$ has the form of
\begin{equation*}
\rho_{1}=\sum_{i=1}^r\lambda_i\varphi_iv_i.
\end{equation*}
In this way, the low-rank model of density functions at time $t_j$ is given by
\begin{equation*}
\label{reduced-pre}
\rho_{j}=\sum_{i=1}^r\lambda^j_i\varphi_iv_i, \quad j\in\mathbb{N}.
\end{equation*}
Let us denote the low-rank approximation of Perron-Frobenius operator as $\rho_{j}=\widetilde{\mathcal{P}}\rho_{j-1}\triangleq\sum_{i=1}^r\lambda_i\varphi_iv_{j-1,i}$, where $v_{j-1,i}$ is the column vector of the matrix $V_{j-1}=\mathbf{W}_{j-1}\Xi^{-1}$. We apply $\widetilde{\mathcal{P}}$ in the Bayesian filter to obtain the low-rank Perron-Frobenius operator filter (lr-PFOF), in which the probability measure satisfies the recursive formula

\begin{equation*}
\label{RPFOF-op}
\mu^N_{j+1}=L_j\widetilde{\mathcal{P}}\mu^N_j, \quad \mu^N_0=\mu_0.
\end{equation*}
To describe the following prediction and analysis steps, we first calculate the weak approximation $P^N$ and get the eigenvalues $\Lambda$ and left eigenvectors $\Xi$ of $P^N$.

\textbf{Prediction} In this step, we give a model decomposition of the prior density $p(x_{j+1}|Y_j)$. First, the $\mathbf{W}_j$ satisfying $p(x_j|Y_j)=\mathbf{W}_j\mathbf{U}$ is obtained from the previous analysis step. Next,  compute the matrix $V_j=\mathbf{W}_j\Xi^{-1}$ and

\begin{equation*}
\widehat{\rho}_{j+1}=p(x_{j+1}|Y_j)=\sum_{i=1}^r\lambda_i\varphi_iv_{j,i}.
\end{equation*}

\textbf{Analysis} In this step, we derive the posterior density $p(x_{j+1}|Y_{j+1})$ via Bayes's formula. Multiply $p(x_{j+1}|Y_j)$ by likelihood function $g_j$ and have

\begin{equation*}
\rho_{j+1}=p(x_{j+1}|Y_{j+1})\propto\sum_{i=1}^r\lambda_i\varphi_iv_{j,i}g_j.
\end{equation*}
To normalize $\rho_{j+1}$, we rewrite the $\widehat{\rho}_{j+1}$. Since $\varphi_i=\xi_i^T\mathbf{U}=\sum_{k=1}^{N}\xi_i^{(k)}\mathds{1}_{\mathbb{B}_k}$, we get

\begin{equation*}
\widehat{\rho}_{j+1} = \sum_{i=1}^r\lambda_i\sum_{k=1}^{N}\xi_i^{(k)}\mathds{1}_{\mathbb{B}_k}v_{j,i}  = \sum_{k=1}^{N}\big(\sum_{i=1}^r \lambda_i \xi_i^{(k)} v_{j,i} \big)\mathds{1}_{\mathbb{B}_k}.
\end{equation*}
Then we multiply by $g_j(x)$ and make a normalization to the weights of $\mathds{1}_{\mathbb{B}_k}$, such that
\begin{equation}
\label{reweight2}
\omega^{(k)}_{j+1}=\widetilde{\omega}^{(k)}_{j+1}/(\sum_{n=1}^N\widetilde{\omega}^{(n)}_{j+1}),
\quad \widetilde{\omega}^{(k)}_{j+1}=\sum_{i=1}^r \lambda_i \xi_i^{(k)} v_{j,i}g_j(x^{(k)}),
\end{equation}
where $x^{(k)}$ is still the mass point  of each box $\mathbb{B}_k$. The posterior density becomes
\begin{equation*}
\rho_{j+1} = \sum_{k=1}^{N}\omega^{(k)}_{j+1}\mathds{1}_{\mathbb{B}_k}=\mathbf{W}_{j+1}\mathbf{U}.
\end{equation*}

\begin{algorithm}[htb]
	\caption{low-rank Perron-Frobenius operator filter}
	\label{alg:2}
	\begin{algorithmic}
     	\STATE \textbf{Offline:}
        \STATE Compute $P^N$ and its eigenvalue $\Lambda$ and left eigenvector $\Xi$. Give the eigenfunction $\boldsymbol{\varphi}=\Xi\mathbf{U}$.
     \end{algorithmic}
     \begin{algorithmic}
        \STATE \textbf{Online:}
     \end{algorithmic}
     \begin{algorithmic}[1]
		\STATE Set $j=0$ and $\rho_{0}=\mathbf{W}_{0}\mathbf{U}$, compute $\omega^{(i)}_0=\frac{\int_{\mathbb{B}_i}\mu_0dx_0}{\mu(\mathbb{B}_i)}$
		\STATE Denote $V_j=\mathbf{W}_j\Xi^{-1}$, compute $\widehat{\rho}_{j+1}=\sum_{i=1}^r\lambda_i\varphi_iv_{j,i}$
        \STATE Define $g_j$ by (\ref{likelihood}), give  $\rho_{j+1}\propto\sum_{i=1}^r\lambda_i\varphi_iv_{j,i}g_j$
        \STATE Normalize weights by (\ref{reweight2}) and obtain $\mathbf{W}_{j+1}$, let $\rho_{j+1} = \sum_{k=1}^{N}\omega^{(k)}_{j+1}\mathds{1}_{\mathbb{B}_k}$.
        \STATE j+1$\rightarrow$ j
        \STATE Go to step 2
	\end{algorithmic}
\end{algorithm}

\begin{rem}
Note that the complex eigenvalues and eigenvectors may appear in the eigendecomposition of the matrix $P^N$. When the stationary distribution $\widetilde{\pi}$ of the system satisfies detailed balance, a symmetrization method is designed in \cite{detail-b} to solve the problem. Since $\rho_0,\rho_1,\cdots$ can be seen as a Markov chain with transition matrix $P^N$, we suppose that $P^N$ satisfies detailed balance with respect to $\widetilde{\pi}$, i.e.,
\[\widetilde{\pi}_iP^N_{ij}=\widetilde{\pi}_jP^N_{ji},\quad \forall i,j\in N.\]
Then $P^N$ can be symmetrized by a similarity transformation
\begin{equation*}
S=\widetilde{\Lambda}P^N\widetilde{\Lambda}^{-1}, \quad {\rm where} \; \widetilde{\Lambda}=\begin{bmatrix} \sqrt{\widetilde{\pi}_1} &   & & \\   & \sqrt{\widetilde{\pi}_2} &  & \\   &  & \ddots &\\  &  &  & \sqrt{\widetilde{\pi}_N} \end{bmatrix}.
\end{equation*}
Here the $S$ is  a symmetric matrix and this can be easily checked by detailed balance equation. It is known that $S$ has a full set of real eigenvalues $\alpha_j\in\mathbb{R}$ and an orthogonal set of eigenvectors $w_j$. Therefore, $P^N$ has the same eigenvalues $\alpha_j$ and real left eigenvectors
\[\psi_j=\widetilde{\Lambda}w_j.\]
\end{rem}

\subsection{Extension to continuous-time filtering problems}
In this subsection, we consider a continuous-time filtering problem, where the state model and observation are  by the following SDEs,

\begin{align}
\label{dynamic1}
\frac{dx}{dt}&=f(x)+\sqrt{\Sigma_c}\frac{dW_t}{dt}, \quad x(t_0)\thicksim \mathcal{N}(m_0,C_0),\\
\label{observation1}
\frac{dz}{dt}&=h(x)+\sqrt{R_c}\frac{dW_t}{dt}, \quad z(0)=0.
\end{align}
 Here $\Sigma_c$ is the covariance of model error and $R_c$ is the  covariance of observation error.  Suppose that the posterior measure $\mu_t$ governed by the continuous-time problem has Lebesgue density  $\rho(\cdot,t):\mathbb{R}^n\mapsto \mathbb{R}^{+}$ for a fixed t. Let $\rho(x,t)=r(x,t)/\int_{\mathbb{R}^n}r(x,t)dx$, where $r$ is the unnormalized density. For a positive definite symmetric matrix $A\in\mathbb{R}^{p\times p}$, we define the weighted inner product $\langle \cdot,\cdot\rangle_A=\langle A^{-\frac{1}{2}}\cdot,A^{-\frac{1}{2}}\cdot\rangle$ on the space $L^2([0,T];\mathbb{R}^p)$. The resulting norm $|\cdot|_{A}=|A^{-\frac{1}{2}}\cdot|.$ In the continuous filtering problem, our interest is to find the distribution of the random variable $x(t)|\{z(s)\}_{s\in[0,t]}$ as the time  $t$ increases. Zakai stochastic partial differential equation (SPDE) is a well-known equation whose solution characterizes the unnormalized density of posterior distribution \cite{Zakai}. The Zaikai equation has the form of

\begin{equation}
\label{Zakai}
\frac{\partial r}{\partial t}=\mathscr{A}_{PF}r+r\bigg\langle h,\frac{dz}{dt}\bigg\rangle_{R_c}.
\end{equation}
The partial differential operator $\mathscr{A}_{PF}$ generates a continuous Perron-Frobenius semigroup $\{\mathcal{P}_t,t\geq 0\}$. Let  $\{\mathcal{Q}^t_s,0\leq s\leq t\}$ be the stochastic semigroup \cite{ds-zakai}  associated with  with the following SDE
\begin{equation}
\label{second-SDE}
\frac{dr'}{dt}=r'\bigg\langle h,\frac{dz}{dt}\bigg\rangle_{R_c}.
\end{equation}
Then the  Zakai equation (\ref{Zakai}) can be approximated by the following Trotter-like product formula

\begin{equation}
\label{it-continuous}
r_{j+1}=Q_{t_j}^{t_{j+1}}\mathcal{P}_{\tau}r_{j},
\end{equation}
where $\tau=t_{j+1}-t_j,\;\forall j\in\mathbb{N}$. For the fixed $\tau$, $\mathcal{P}_{\tau}$ is still denoted by $\mathcal{P}$. By the reference  \cite{ds-zakai}, the $Q_{t_j}^{t_{j+1}}$ describes the solution of the equation (\ref{second-SDE}), i.e.,
\[
Q_{t_j}^{t_{j+1}}r(x)={\rm exp}\bigg(\langle h(x),z_{j+1}-z_{j}\rangle_{R_c}-\frac{\tau}{2}|h(x)|^2_{R_c}\bigg)r(x).
\]
 With the discrete scheme (\ref{it-continuous}), we utilize the Perron-Frobenius operator to solve Zakai equation, rather than using Fokker-Planck operator $\mathscr{A}_{PF}$.  Thus, we discretize $\mathcal{P}$ by Ulam's method and project the density function onto $V_N$.  Let $P^{N}$ be the discretization of $\mathcal{P}$. Let $\mathbf{W}_{j}$ and $\mathbf{W}_{j+1}$ be the weights vectors with respect to $\pi_N r_{j}$ and $\pi_N r_{j+1}$. Denote $g_j^c(x)={\rm exp}\big(\langle h(x),z_{j+1}-z_{j}\rangle_{R_c}-\frac{\tau}{2}|h(x)|^2_{R_c}\big)$ and
\[
\mathbf{G}_j=\begin{bmatrix}
g_j^c(x^{(1)})\\
g_j^c(x^{(2)}) \\
\vdots \\
g_j^c(x^{(N)})
\end{bmatrix},
\]
where $x^{(i)}$ is the mass point of $\mathbb{B}_i$. Then the transition of density functions turns into a map of the weights,
\begin{equation*}
\label{PF-Zakai}
\mathbf{W}_{j+1}=\mathbf{G}_j\odot\big(\mathbf{W}_{j}P^{N}\big).
\end{equation*}
Here $\odot$ denotes Hadamard product. In this case, the PFO is extended to the continuous-time filtering problem to estimate the posterior density function.

\section{Comparison with particle filter}
\label{sect4}

Particle filter (PF) \cite{PF1,PF2}  is an important filtering method to sequentially approximate the true posterior filtering distribution $p(x_j|Y_j)$ in the limit of a large number of  particles. In practice, we approximate the probability density by a combination of locations of particles and weights associated with Dirac functions. Particle filter proceeds by varying the weights and determining the particle Dirac measures.  It is able to take care of non-Gaussian and nonlinear models. In this section, we will compare the computational accuracy and differences between PFOF and PF.

Accordingly, we define $\mu_j^N$  as the posterior empirical measure on $\mathbb{R}^{N}$ approximating truth posterior probability measure $\mu_j$  and define $\widehat{\mu}_{j}^N$ on $\mathbb{R}^{N}$ as the approximation  of the prior probability measure $\widehat{\mu}_{j}$. Let
\begin{equation*}
\mu_j\approx \mu_j^N:=\sum_{n=1}^{N}\omega_j^{(n)}\delta_{x_j^{(n)}},
\quad \widehat{\mu}_{j+1}\approx \widehat{\mu}_{j+1}^N:=\sum_{n=1}^{N}\widehat{\omega}_{j+1}^{(n)}\delta_{\widehat{x}_{j+1}^{(n)}},
\end{equation*}
where $x_j^{(n)}$ and  $\widehat{x}_{j+1}^{(n)}$  are particle positions, and   $\omega_j^{(n)}>0$,  $\widehat{\omega}_{j+1}^{(n)}>0$  are the associated  weights satisfying
$\sum_{n=1}^{N}\omega^{(n)}_j=1,\;\sum_{n=1}^{N}\widehat{\omega}_{j+1}^{(n)}=1.$ The empirical distribution is completely determined  by particle positions and weights. The objective of particle filter is to calculate  the update
$
\{x_j^{(n)},\omega_j^{(n)}\}\rightarrow \{\widehat{x}_{j+1}^{(n)},\widehat{\omega}_{j+1}^{(n)}\}$ and  $\{\widehat{x}_{j+1}^{(n)},\widehat{\omega}_{j+1}^{(n)}\}\rightarrow \{x_{j+1}^{(n)},\omega_{j+1}^{(n)}\}$, which define the prediction step and analysis step,  respectively. Monte-Carlo sampling is used  to determine particle positions in the prediction and Bayesian rule is used to  update of the weights in the analysis.

\textbf{Prediction}   In this step, the prediction phase is approximated by the Markov chain $\{\Psi(x_j)\}_{j\in\mathbb{N}}$ with transition kernel $p(x_j,x_{j+1})=p(x_{j+1}|x_j)$. We draw $\widehat{x}^{(n)}_{j+1}$ from the kernel $p$  started from $x^{(n)}_j$, i.e., $\widehat{x}^{(n)}_{j+1}\thicksim p(x_j^{(n)},\cdot)$.  We keep the weights unchanged so that $\widehat{\omega}_{j+1}^{(n)}=\omega_j^{(n)}$, and obtain  the prior probability measure
\begin{equation*}
\widehat{\mu}_{j+1}^N = \sum_{n=1}^{N}\omega_j^{(n)}\delta_{\widehat{x}_{j+1}^{(n)}}.
\end{equation*}

\textbf{Analysis} In this step, we apply Bayes's formula to approximate the posterior probability measure. To do this, we fix the position of the particles and update the weights. With the definition of  $g_j(x)$ in (\ref{likelihood}), we have the empirical posterior distribution
\begin{equation*}
\mu_{j+1}^N = \sum_{n=1}^{N}\omega_{j+1}^{(n)}\delta_{\widehat{x}_{j+1}^{(n)}},
\end{equation*}
where
\begin{equation}
\label{reweight-pf}
\omega_{j+1}^{(n)}=\widetilde{\omega}_{j+1}^{(n)}/(\sum_{n=1}^N\widetilde{\omega}_{j+1}^{(n)}),
\quad \widetilde{\omega}_{j+1}^{(n)}=g_j(\widehat{x}_{j+1}^{(n)})\omega_j^{n}.
\end{equation}
The first equation in (\ref{reweight-pf}) is a normalization. Sequential Importance Resampling (SIR) particle filter is a  basic  particle filter and  shown in Algorithm \ref{alg:1}. A resampling step is introduced in the algorithm. In this way, we can deal with the initial measure $\mu_0$ when it is not  a combination of Dirac functions. We can also deal with the case when some of the particle weights are close  to 1. The algorithm shows  that each particle moves according to the underlying model and is reweighted according to the likelihood. By the iteration of Bayesian  filtering , we rewrite the particle filter approximated by  the form
\begin{equation}
\label{PF-op}
\mu^N_{j+1}=L_jS^N\mathcal{P}\mu^N_j, \quad \mu^N_0=\mu_0,
\end{equation}
where the operator $S^N$ is defined as follows:

\begin{equation*}
\label{S-O}
(S^N\mu)(dx)=\frac{1}{N}\sum_{n=1}^{N}\delta_{x^{(n)}}(dx), \quad x^{(n)}\thicksim \mu \quad \rm i.i.d..
\end{equation*}

\begin{algorithm}
	\caption{Sequential Importance Resampling particle filter}
	\label{alg:3}
	\begin{algorithmic}[1]
		\STATE Set $j=0$ and $\mu_0^N(dx_0)=\mu_0(dx_0)$
		\STATE Draw $x_j^{(n)}\thicksim \mu_j^N, n=1,\cdots,N$
		\STATE Set $\omega_j^{(n)}=1/N, n=1,\cdots,N$, redefine $\mu_j^N:=\sum_{n=1}^{N}\omega_j^{(n)}\delta_{x_j^{(n)}}$
		\STATE Draw $\widehat{x}_{j+1}^{(n)}\thicksim p(x_j^{(n)},\cdot)$
        \STATE Define $\omega_{j+1}^{(n)}$ by $(\ref{reweight})$ and $\mu_{j+1}^N = \sum_{n=1}^{N}\omega_{j+1}^{(n)}\delta_{\widehat{x}_{j+1}^{(n)}}$
        \STATE j+1$\rightarrow$ j
        \STATE Go to step 2
	\end{algorithmic}
\end{algorithm}
By (\ref{PF-op}), we find that the randomness for the probability measure is caused by the sampling operator $S^N$ and the convergence of particle filter depends on the number of particles.  The particle filter does recover the truth posterior distribution as the number of particles tends to infinity \cite{PF-cov2}. The following theorem gives a convergence result for PF.

\begin{thm}{\rm (Theorem 4.5 in \cite{DA-ma})}
\label{Thm2-PF}
Let  $m$  be the number of particles and $\mu_j^m$  the approximation measure in SIR particle  filter. Assume that $\kappa\in(0,1]$ is the constant defined in Lemma \ref{lem2}, then the total-variance distance between $\mu_J^m$ and $\mu_J$ is estimated by
\begin{equation}
d(\mu_J^m,\mu_J)\leq \sum_{j=1}^{J}(2\kappa^{-2})^j\frac{1}{\sqrt{m}}.
\end{equation}
\end{thm}
Let $J$ be fixed in Theorem \ref{Thm1-PFOF} and Theorem \ref{Thm2-PF}. We find that the convergence rate of particle filter depends on the number of particles $m$. Similarly, the  convergence of PFOF is determined by the number of blocks $N$ used  in the Ulam's method. When $N=m$, i.e., the same  number of basis functions in the two methods, the rate of convergence   is $\mathcal{O}(\frac{1}{N})$ in PFOF and  $\mathcal{O}(\frac{1}{\sqrt{m}})$ in SIR particle filter. The analysis shows  that PFOF converges faster than the particle filter.


Sampling from high-dimensional and complex transition kernels is difficult to realize in PF. The PFOF avoids the sampling and uses a data-driven approximation instead, which requires short-term path simulations rather than the form of transition density.  Particle degeneracy is also a significant issue. As the number of effective particles decreases gradually, the efficiency of the particle filter becomes worse.

It is known  that particle filter is inefficient for high-dimensional models because of degeneracy. So the accurate estimate of posterior PDF requires a great number of particles that scales exponentially with the size of the system. In addition to resampling,  adding jitter and localisation are effective modifications to solve the problem.  The PFOF also has the ``curse of dimensionality'' problem in high dimensions as the partition scale expansion. One solution to circumvent this problem is the sparse Ulam method. The low-rank Perron-Frobenius operator filter  can enhance  the  efficiency of filtering problems.

\section{Numerical results}

\label{sect5}
In this section, we present some numerical examples  for filtering problems using the proposed PFOF. The system dynamics is unknown  and some  observations are given in the filtering problems. The PFOF and lr-PFOF are implemented to estimate posterior PDFs of the stochastic filtering problems. In Section 5.1, we consider  an Ornstein-Uhlenbeck (O-U) process to identify  the Gaussian PDF of the system and estimate its posterior PDFs with observations known. In Section 5.2, we consider a nonlinear filtering problem governed by   Bene$\breve{s}$ SDE, and  estimate the non-Gaussian posterior PDFs. In Section 5.3, we consider  a continuous-time filtering problem, which is a classical chaotic system  Lorenz'63 model with observations, to model posterior density of the state. We compare the proposed PFOF/lr-PFOF with particle filter and Extended Kalmn filter (ExKF). Numerical results show that PFOF achieves  a better  posterior PDF estimates than PF, and a more accurate state estimates than ExKF.

\subsection{O-U process}
Let us consider an O-U process, which is a one-dimensional linear dynamical system,
\begin{equation*}
\label{O-U}
dx_t=-\lambda x_tdt+dW_t, \quad x(t_0)\thicksim \mathcal{N}(m_0,C_0),
\end{equation*}
where $\lambda>0$ and $W_t$ is a standard Brownian motion. We now consider a state-space model formed by a discretization of the O-U process and the  discrete observations of the state as follows,
\begin{equation}
\label{O-U-observ}
\left\{
\begin{aligned}
&x(t_{k+1})={\rm exp}(-\lambda\Delta t_k)x(t_k)+q_k, \quad q_k\thicksim \mathcal{N}(0,\Sigma_k), \\
&y(t_k)=Hx(t_k)+r_k,\quad  r_k\thicksim \mathcal{N}(0,R),&
\end{aligned}
\right.
\end{equation}
where $\Sigma_k={\rm exp}(-2\lambda\Delta t_k)$, $H=I$ and $R=\sigma^2$. The parameters are given by $\lambda = 1/2$, $m_0=2$, $C_0=0.1$ and $\sigma=1$. To apply PFOF,   we compute  Perron-Frobenius operator $P_{\tau}$ using Ulam's method with time step $\tau=0.1$ and obtain an approximation form $P_{\tau}^N\in\mathbb{R}^{N\times N}$ of $\mathcal{P}_{\tau}$. We take the  phase space of $x_t$ is $[-6, 6]$ and divide it  into $N=100$ grids, and each interval  $[z_k,z_{k+1}],\;k=0,\cdots,N-1$,  defines a box $\mathbb{B}_k$. We define an indicator function $\mathds{1}_{\mathbb{B}_k}(x)$ on each $\mathbb{B}_k$ and  randomly choose $n=100$ sample points in the box to calculate $P_{\tau}^N$. Given initial Gaussian distribution $\mathcal{N}(2,0.1)$, we rewrite $\mu_0$ as a vector $W_0$, which  denotes  the coefficients of $\mu_0^N$. The $P_{\tau}^N$  acts on the weight vector to estimate probability value of $x_t$ on each $\mathbb{B}_k$, i.e., $\mathbb{P}(x_t\in\mathbb{B}_k)$,    $t=q\tau,\;q=0,1,2\cdots$. Thus, we get the discrete probability density function (PDF) of $x_t$ at $t$.   The simulation  PDFs  at different times are shown  in the left column of Figure $\ref{grid1}$.  By the figure, we see that the PDFs estimated by PFO are close to the truth. By this way, the PDF is computed  without solving Fokker-Planck equation and the estimation of PDF is actually the prior density in the model  (\ref{O-U-observ}).

\begin{figure}[htbp]
  \centering
  \includegraphics[width=0.8\textwidth]{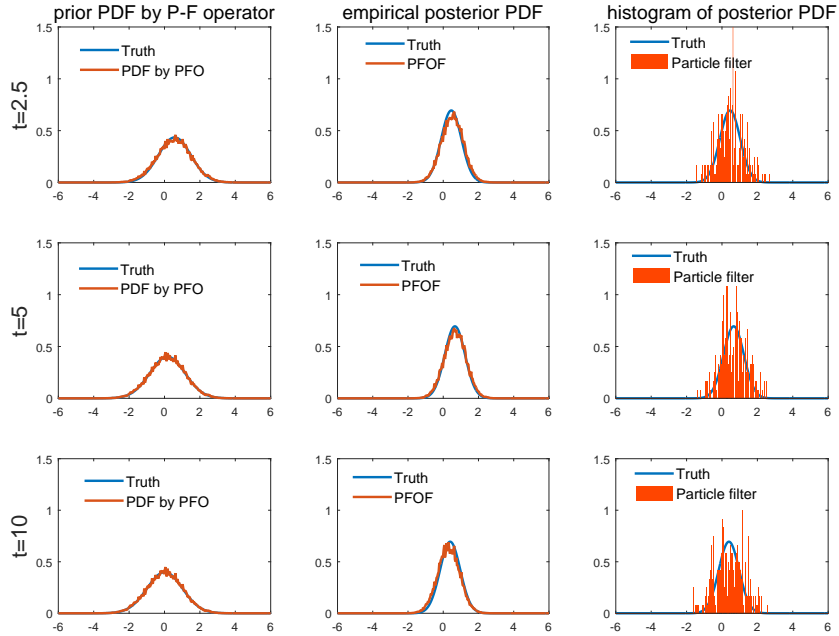}
  \caption{\footnotesize{The prior PDF estimated by PFO (left column), posterior PDF by PFOF (middle column) and posterior PDF by  particle filter (right column) at different times.}}\label{grid1}
\end{figure}

Then we compute posterior probability density of the state-space model (\ref{O-U-observ}). We set $N=500$ and $n=100$. The posterior probability density function is estimated  by Algorithm \ref{alg:1} and the results are displayed in the middle column of Figure $\ref{grid1}$. From Figure $\ref{grid1}$, we find  that the empirical posterior PDFs estimated by PFOF are close to the Gaussian posterior densities. To make comparison with PFOF,  the  particle filter is also used for the filtering problem. In the particle filter, $500$ particles are drawn  randomly to generate Dirac measure and  construct empirical measure. Thus, the number of basis functions is equal to each other in the two methods. Figure $\ref{grid1}$ clearly shows  that the empirical PDF calculated by PFOF is more accurate than that by PF.  The numerical results support Theorem \ref{Thm1-PFOF} and Theorem \ref{Thm2-PF}.


\subsection{Bene$\breve{s}$-Daum filter}
In this subsection, we apply PFOF to a nonlinear filtering problem, whose state-space model is defined by the Bene$\breve{s}$ stochastic difference equation,
\begin{equation}
\label{Benes}
dx_t={\rm tanh}(x_t)dt+dW_t,
\end{equation}
with initial condition $x_0=0$. Refer to \cite{SDE}, the probability density function of the equation (\ref{Benes}) is given by
\[
p(x(t))=\frac{1}{\sqrt{2\pi t}}\frac{{\rm cosh}(x(t))}{{\rm cosh}(x_0)}{\rm exp}\big(-\frac{t}{2}\big){\rm exp}\big(-\frac{1}{2t}(x(t)-x_0)\big).
\]
We take  the phase space $[-15,  15]$ and uniformly divide it  into $100$ ($N=100$)  grids $[z_{k},z_{k+1}],\; k=0,...,N-1$, each of which  corresponds  to a box $\mathbb{B}_k$.  The  Ulam's method  is used to approximate PFO. The time step  is set as $\tau=0.5$ and the number of random sample points  $m=400$. The predicted PDF of $x_t$ at $t=1$, $t=2.5$ and $t=5$ are  shown in Figure $\ref{grid2}$. The PDFs are separately estimated by discretized PFO matrix $P^N\in \mathbb{R}^{100\times100}$ and  low-rank approximation of PFO with truncation $r=30$. We see that PDFs at $t=2.5$ and $t=5$ have two modes and the PFO can fairly approximate the two modes.

\begin{figure}[htbp]
  \centering
  \includegraphics[width=0.8\textwidth]{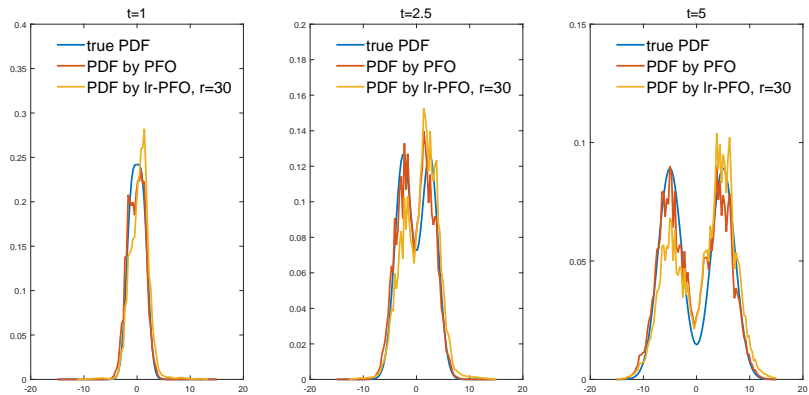}
  \caption{\footnotesize{The PDF estimated by PFO and low-rank model at different times.}}\label{grid2}
\end{figure}

First we want to calculate the truth  posterior filtering distribution of the model (\ref{Benes}) subject to observation. In this example,  the observation model satisfies
\begin{equation}
\label{Benes-ob}
p(y_k|x(t_k))=\mathcal{N}(y_k|x(t_k),\sigma^2).
\end{equation}
According  to \cite{SDE} (Chapter 10.5), the transition density of the Bene$\breve{s}$ SDE is given by
\begin{equation*}
p(x(t_k)|x(t_{k-1}))=\frac{1}{\sqrt{2\pi\Delta t_{k-1}}}\frac{{\rm cosh}(x(t_k))}{{\rm cosh}(x(t_{k-1}))} {\rm exp}(-\frac{1}{2}\Delta t_{k-1})\times {\rm exp}\bigg(-\frac{1}{2\Delta t_{k-1}}(x(t_k)-x(t_{k-1}))^2\bigg),
\end{equation*}
where $\Delta t_{k-1}=t_k-t_{k-1}$. If we assume that the filtering solution at time $t_{k-1}$ is of the form
\begin{equation*}
p(x(t_{k-1})|y_{1:k-1})\propto {\rm cosh}(x(t_{k-1})){\rm exp}\bigg(-\frac{1}{2P_{k-1}}(x(t_{k-1})-m_{k-1})^2\bigg)
\end{equation*}
for given  $m_{k-1}$ and $P_{k-1}$. Then we use the Chapman-Kolmogorov equation and give the prior density
\begin{equation*}
\label{b_prior}
p\big(x(t_k)|y_{1:k-1}\big)\propto {\rm cosh}\big(x(t_k)\big){\rm exp}\bigg(-\frac{1}{2P_k^{-}}(x(t_k)-m_k^{-})^2\bigg),
\end{equation*}
where
\[ m_k^{-}=m_{k-1},
\]

\[P_k^{-}=P_{k-1}+\Delta t_{k-1}.
\]
The $m_k^{-}$ and $P_k^{-}$ are sufficient statistics representing prior density functions. By Bayes' formula, the posterior density of $x(t_k)$ is given by
\begin{equation}
\label{b_posterior}
p\big(x(t_k)|y_{1:k}\big)\propto {\rm cosh}\big(x(t_k)\big){\rm exp}\bigg(-\frac{1}{2P_k}\big(x(t_k)-m_k\big)^2\bigg),
\end{equation}
where the equations of parameters $m_k$ and $P_k$ in the posterior density satisfy
\[ m_k=m_k^{-}+\bigg(\frac{P_k^{-}}{P_k^{-}+\sigma^2}\bigg)(y_k-m_k^{-}),
\]

\[P_k^{-}=P_{k-1}+\Delta t_{k-1}.
\]
Thus, the reference posterior distribution is defined by (\ref{b_posterior}).

To apply PFOF to the nonlinear filtering problem, we make  a finer division of the phase interval $[-15,15]$ to obtain $400$ boxes. Besides, we choose enough sample points in Ulam's method to reduce error of Monte-Carlo  as much as possible. The observations $y_k$ are artificially obtained by simulating the underlying model (\ref{Benes}) and adding noise according to (\ref{Benes-ob}), where $\sigma=1$. The observable interval is $[0,5]$ with a time step $\Delta t_k=0.1$. The initial distribution for the filtering process  is chosen to be $m_0=0$, $P_0=2$. Particularly, we also use the particle filter as a comparison.  In the prediction, we are not allowed to draw sample points directly  because of a complex transition probability density function. We use Acceptance-Rejection method to resolve the issue. We first show the results of posterior mean estimated by PFOF and lr-PFOF (r=40) in Figure $\ref{grid3}$, together with truth and observations. The mean is obtained by averaging the posterior distribution of PFOF/lr-PFOF and it is close to the truth as the figure shows.

\begin{figure}[htbp]
  \centering
  \includegraphics[width=0.5\textwidth]{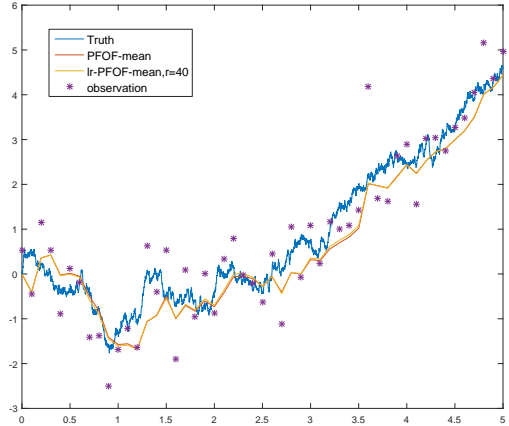}
  \caption{\footnotesize{The mean estimated by PFOF and lr-PFOF.}}\label{grid3}
\end{figure}

The posterior densities estimated by PFOF, lr-PFOF and particle filter are shown in Figure $\ref{grid4}$ together with the truth. The truncation parameters in lr-PFOF are separately set as $r=10$, $r=20$ and $r=40$. The estimation accuracy  of lr-PFOF gradually improves as the number of truncation basis functions increases, and achieves almost the same as PFOF when $r=40<N=400$. Although the number of basis functions is  the same in both PFOF and particle filter, there exit clear difference  between the two methods. The results show that the accuracy of PFOF is higher than that of particle filter  in the non-Gaussian and nonlinear filtering problem. This further confirms the theoretical analysis in Section \ref{sect3}.  As shown in Table \ref{table1},  both PFOF and lr-PFOF use less  CPU-time than SIR particle filter does. Actually, the CPU-time  in particle filter is mainly from  Acceptance-Rejection sampling. From the table, it can be seen that lr-PFOF can reduce online computation  time comparing to PFOF.

\begin{figure}[htbp]
  \centering
  \includegraphics[width=1\textwidth]{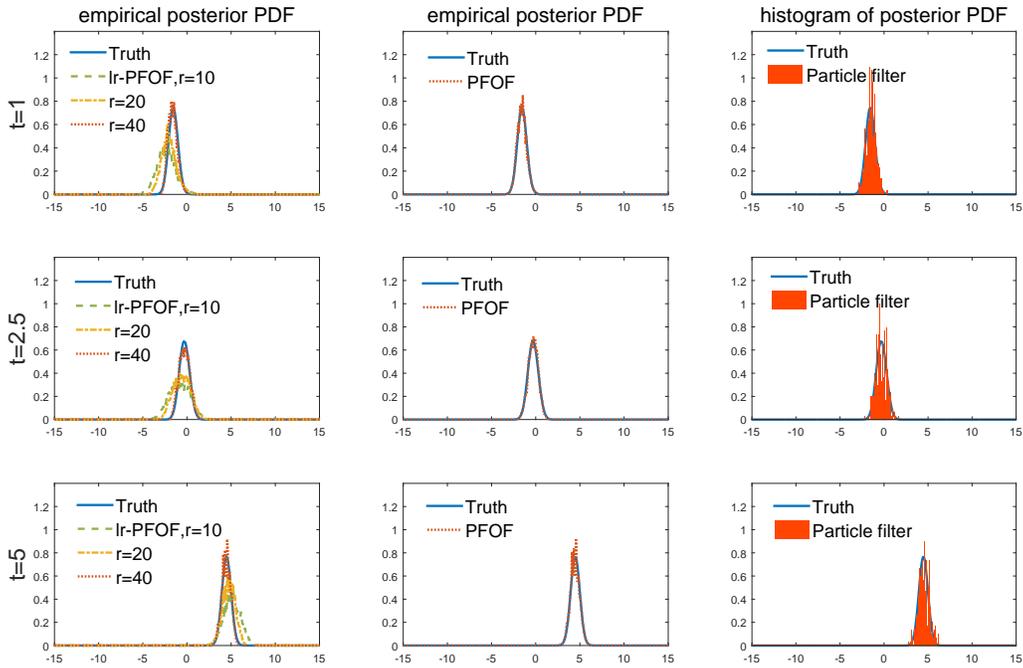}
  \caption{\footnotesize{The posterior PDF by lr-PFOF (left column), PFOF (middle column) and particle filter (right column) at different times.}}\label{grid4}
\end{figure}

\begin{table}[htbp]
	\centering
	\setlength{\abovecaptionskip}{0cm}
	\setlength{\belowcaptionskip}{0.2cm}
	\caption{CPU-time (seconds) for posterior PDF with different methods.}
	\label{sec4_exper1:morlet_table_hyper}
	\scalebox{0.8}{
		\begin{tabular}{cccccc}
			\toprule
			Methods & PFOF & lr-PFOF (r=10) & lr-PFOF (r=20) & lr-PFOF (r=40) & particle filter \\
			\midrule
			& & & & & \\[-6pt]
			offline &0.1599 &0.2536 &0.2649 &0.2689 &\multirow{2}*{6906.9486} \\
			\cline{1-5}
			& & & & & \\[-6pt]
			online &0.0673 &0.0150 &0.0431 &0.0641 & \\
			\bottomrule
		\end{tabular}
	}\label{table1}
\end{table}
\subsection{Lorenz'63 model}
Lorenz developed a mathematical model for atmospheric convection in 1963. The Lorenz'63 model is the simplest continuous-time system to exhibit sensitivity to initial conditions and chaos, and it is popular example used  for data assimilation. For some parameters and initial conditions, the system may perform a chaotic behaviour. The model consists of  three coupled nonlinear ordinary differential equations with the solution $v=(v_1,v_2,v_3)\in\mathbb{R}^3$. We consider the Lorenz'63 model with additive white noise,
\[
\left\{
\begin{aligned}
\frac{dv_1}{dt}&=a(v_2-v_1)+\sigma_1\frac{dW_1}{dt} \\
\frac{dv_2}{dt}&=-av_1-v_2-v_1v_3+\sigma_2\frac{dW_2}{dt} \\
\frac{dv_3}{dt}&=v_1v_2-bv_3-b(r+a)+\sigma_3\frac{dW_3}{dt}\\
v(0)&\thicksim\mathcal{N}(m_0,C_0),
\end{aligned}
\right.
\]
where $W_j$ are Brownian motions assumed to be independent.  We use the classical parameter values $(a,b,r)=(10,\frac{8}{3},28)$ and set $\sigma_1=\sigma_2=\sigma_3=2$. The initial mean $m_0$ is given by $(0,0,0)$ and covariance matrix is an identity matrix $I_3\in\mathbb{R}^{3\times 3}$.  We give the continuous observation $z(t)$, which is governed by a SDE
\[
\left\{
\begin{aligned}
\frac{dz}{dt}&=h(v)+\gamma\frac{dW_z}{dt}\\
z(0)&=0,
\end{aligned}
\right.
\]
with $\gamma=0.2$. The purpose of this example is to explore the performance of PFOF in continuous-time filtering problems. We compare the assimilation  results based on Perron-Frobenius operator and continuous-time Extended Kalman filter. The posterior means estimated by the  two methods are shown in Figure \ref{grid5} and Figure \ref{grid7}. The two figures are corresponding to different  observations $h(v)=Hv$,  where the former is determined by $H=[0,1,0]$ and the latter is determined by $H=[0,0,1]$. In particular, we find that the choice of observations in Lorenz models is quite influential, especially for ExKF. The stability of  ExKF significantly depends on the observation. Because the insufficient observations may keep the filter away from the truth and cause significant model error, and it may easily lead to the numerical instability   once the deviation occurs. However, the results of reconstruction by PFOF much less  affected by observation model, so the method shows much better  robustness than ExKF.

Figure \ref{grid6}  shows  the consequence of mean-square error with  $v_2$ or $v_3$ as the different observation. For ExKF, we find that there is a large error in estimating mean by ExKF when the third component $v_3$  is observed. To better visualize  the results, we compare the trajectories of mean obtained by PFOF and ExKF in Figure \ref{grid8} together with truth. We find the trajectory  mean of PFOF agrees with the truth more than the the trajectory  mean of ExKF.

\begin{figure}[htbp]
  \centering
  \includegraphics[width=1\textwidth]{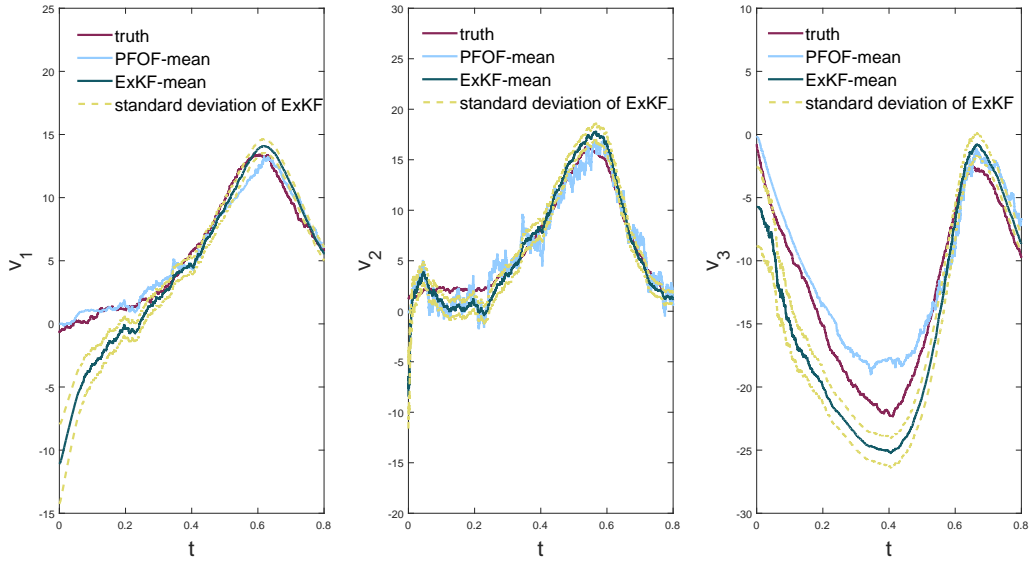}
  \caption{\footnotesize{The posterior mean of each component by ExKF and PFOF in Lorenz'63 model with continuous observation. The component $v_2$ is observed.}}\label{grid5}
\end{figure}

\begin{figure}[htbp]
  \centering
  \includegraphics[width=0.6\textwidth]{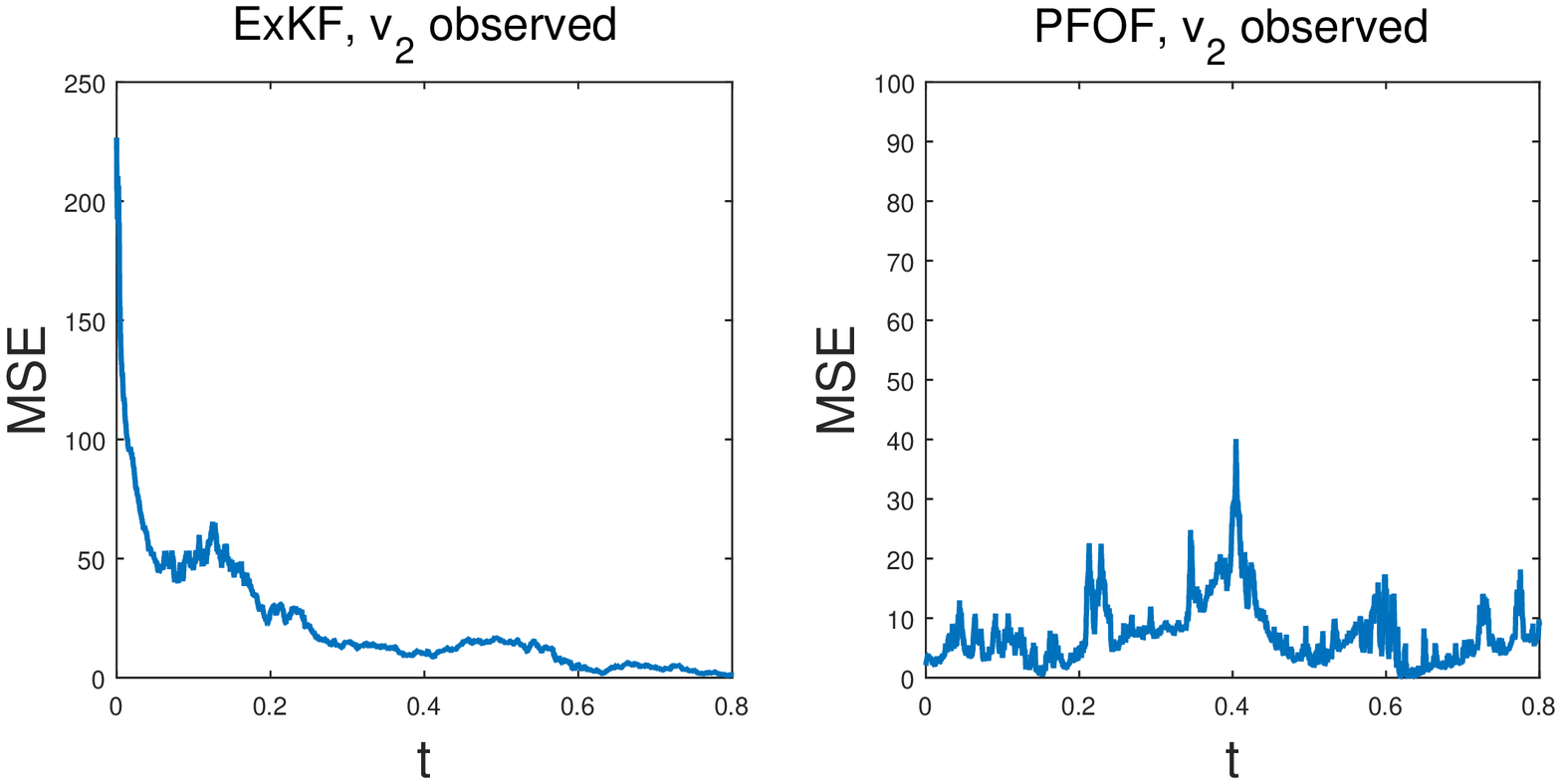}\\
   \includegraphics[width=0.6\textwidth]{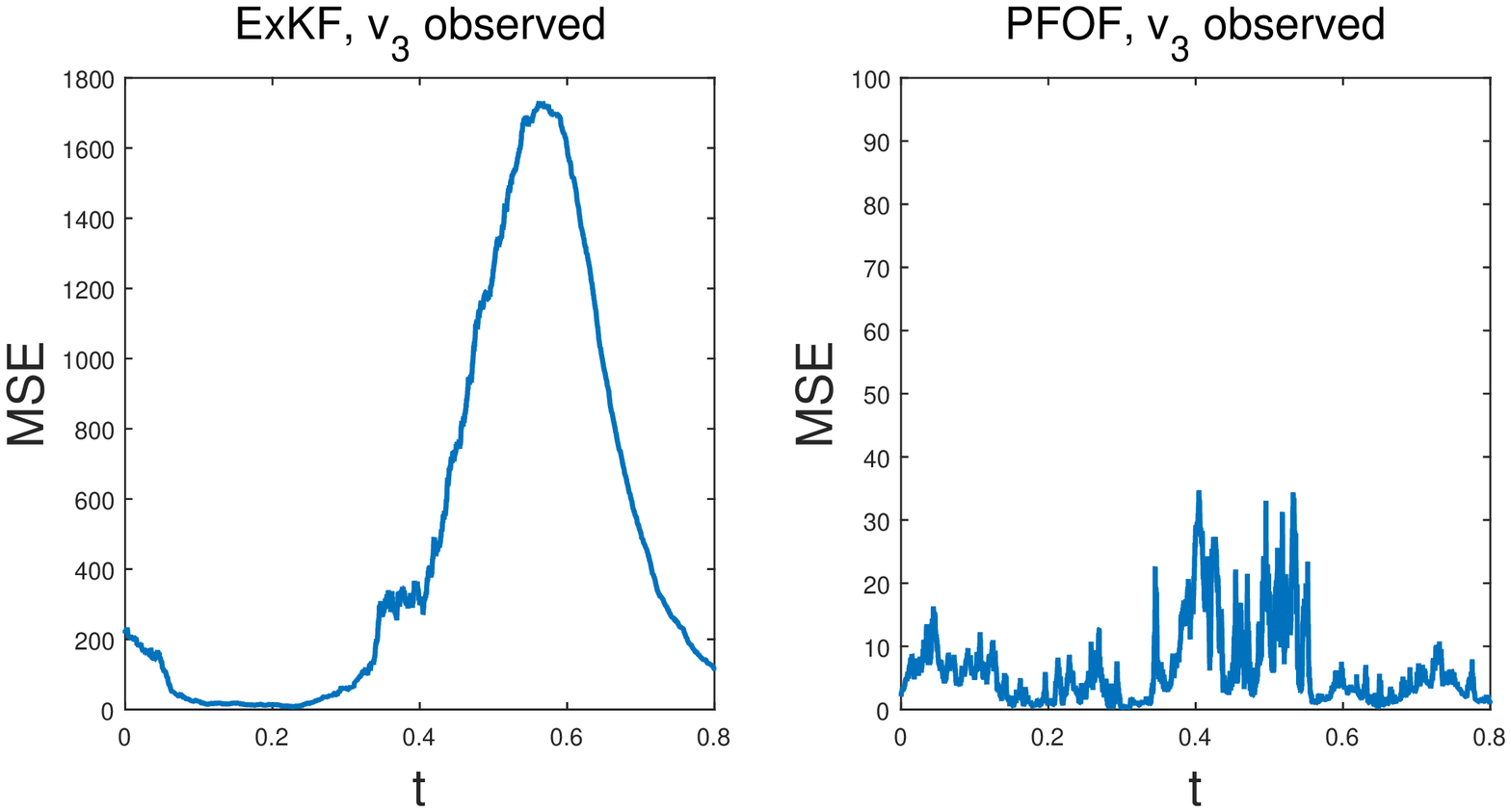}
  \caption{\footnotesize{The mean-square error $\|v(t)-m(t)\|_2^2$ of filters.}}\label{grid6}
\end{figure}

\begin{figure}[htbp]
  \centering
  \includegraphics[width=1\textwidth]{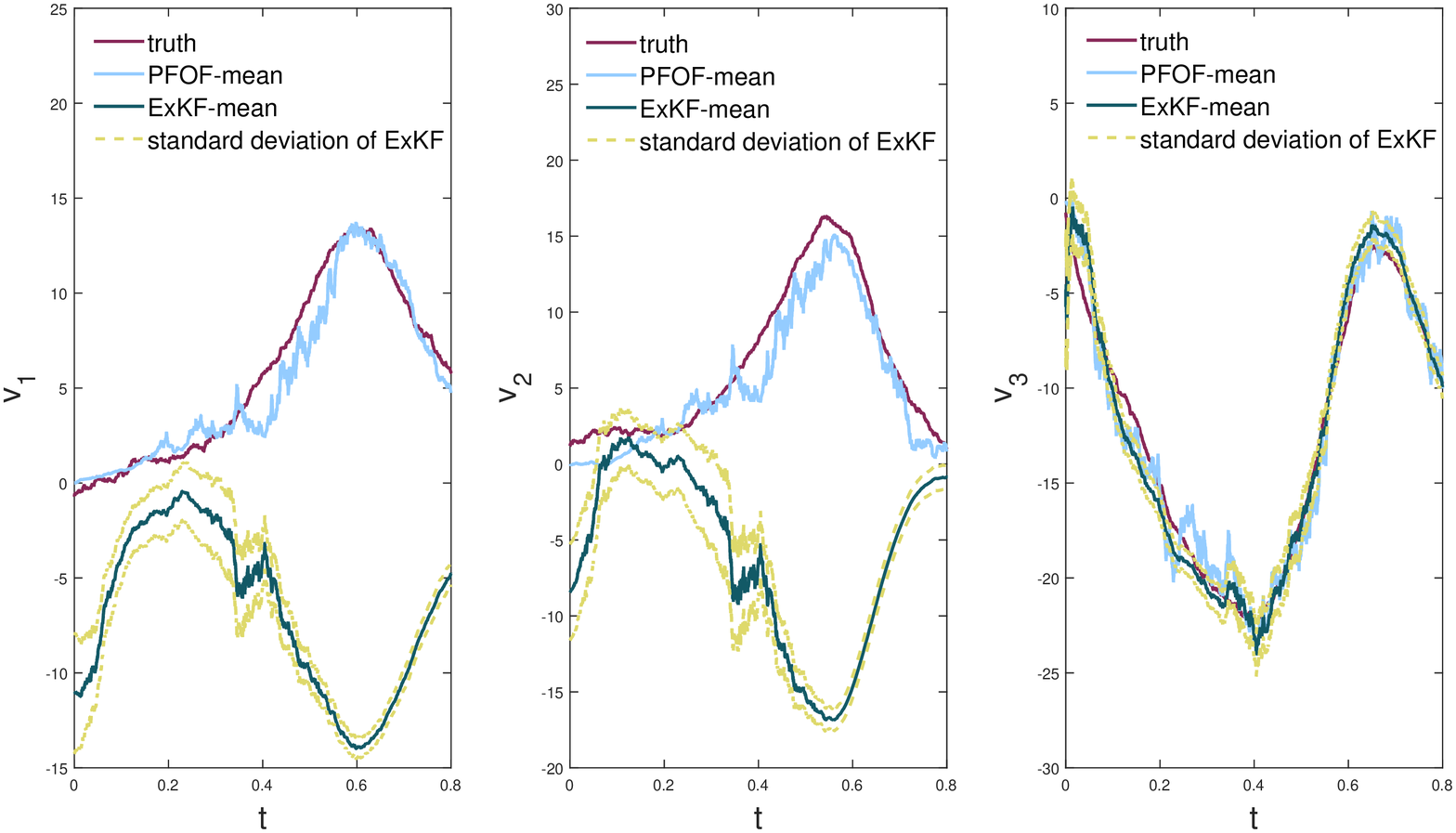}
  \caption{\footnotesize{The posterior mean of each component by ExKF and PFOF in Lorenz'63 model with continuous observation. The component $v_3$ is observed.}}\label{grid7}
\end{figure}

\begin{figure}[htbp]
  \centering
  \includegraphics[width=0.8\textwidth]{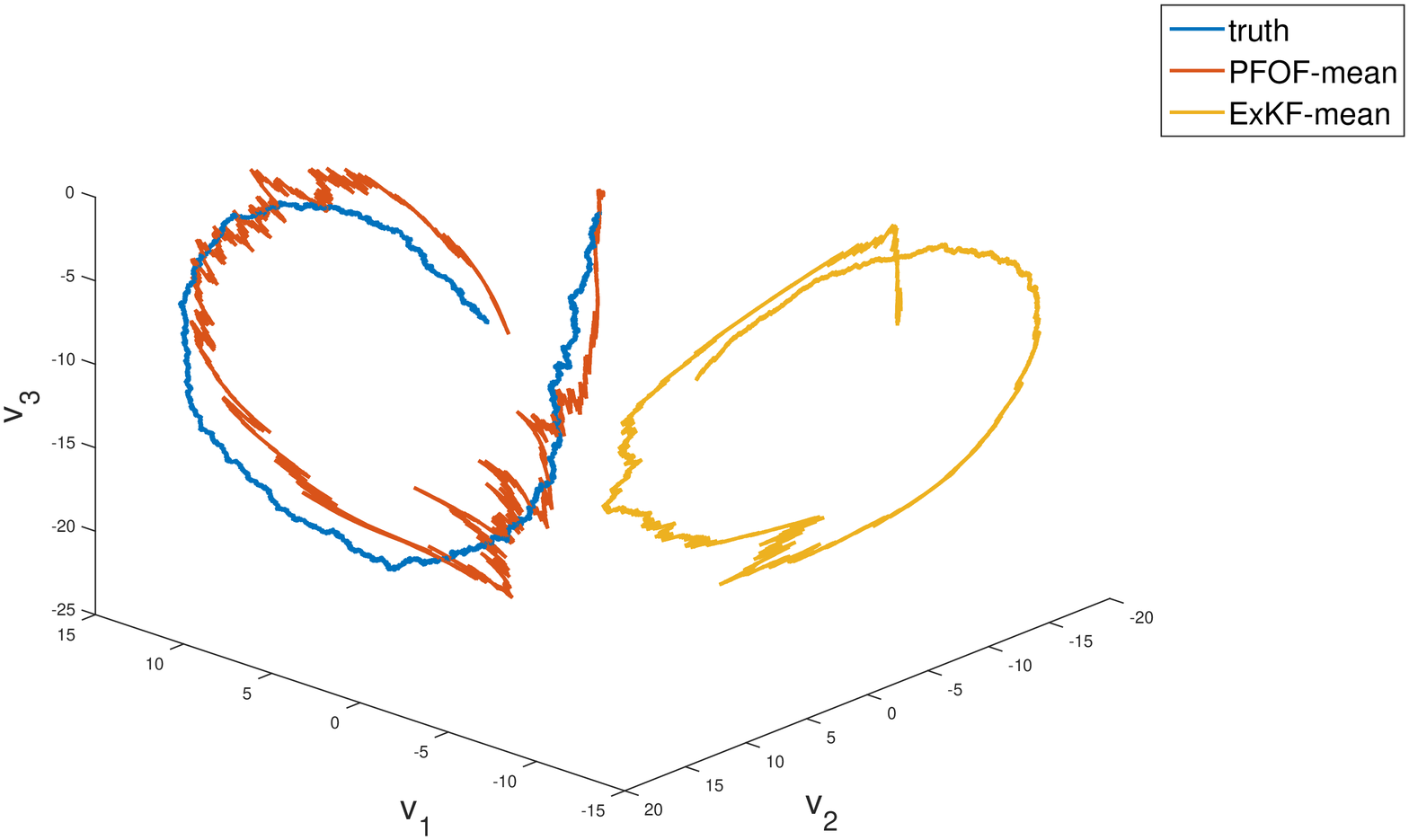}
  \caption{\footnotesize{The trajectories of mean by PFOF and ExKF.}}\label{grid8}
\end{figure}

For $v_3$ as an observation, the one-dimensional and two-dimensional marginal probability distributions are displayed in Figure $\ref{grid9}$. The figure aims to intuitively describe distribution of the single value and correlation  of the different components. As shown in the figure, one-dimensional marginal distributions of the observed component are closer to Gaussian distributions than the other two components. This phenomenon reflects that when a component is used as an observation, its mean estimates will be more accurate than the other unobserved  components.


\begin{figure}[htbp]
  \centering
  \includegraphics[width=0.6\textwidth]{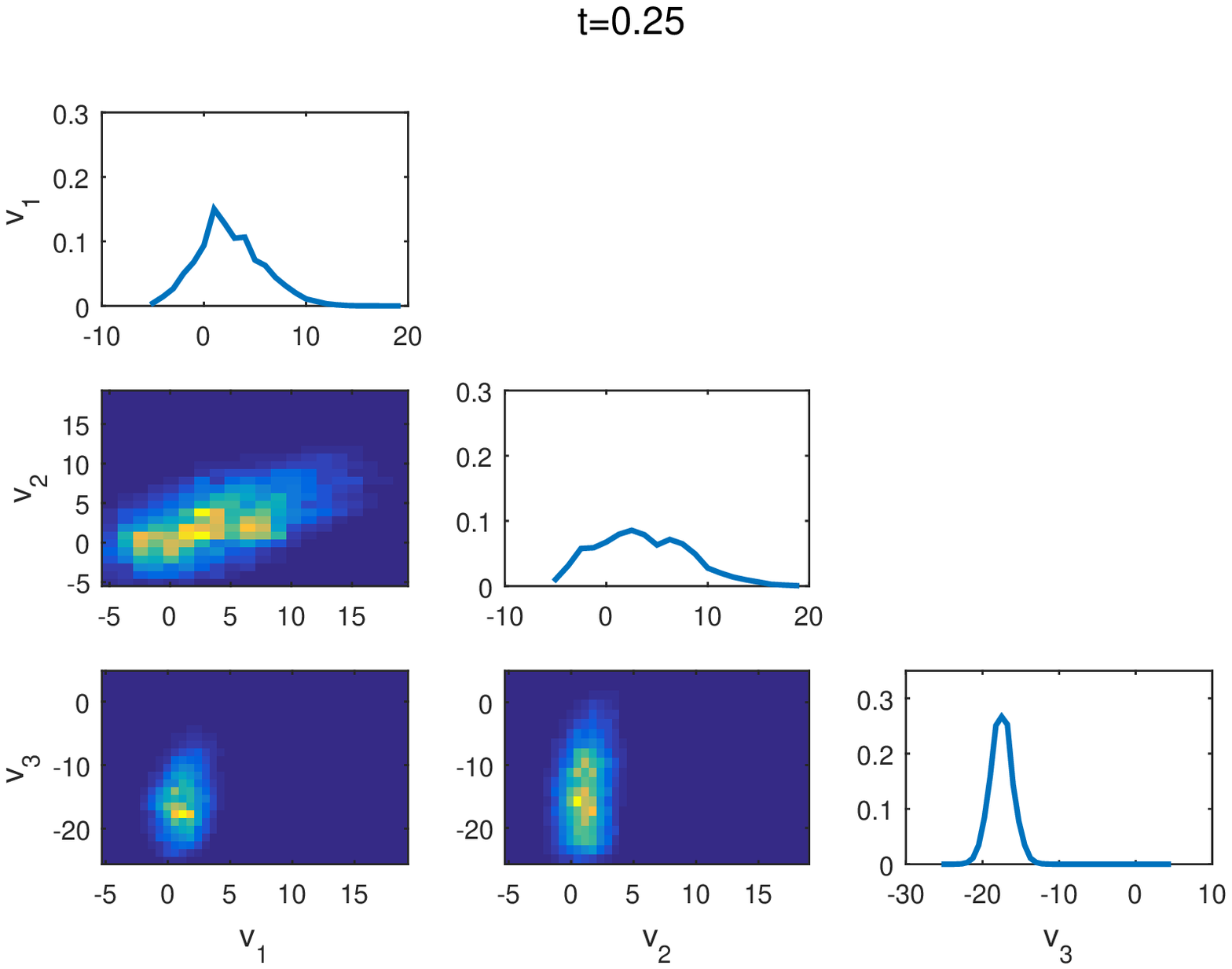}
   \includegraphics[width=0.6\textwidth]{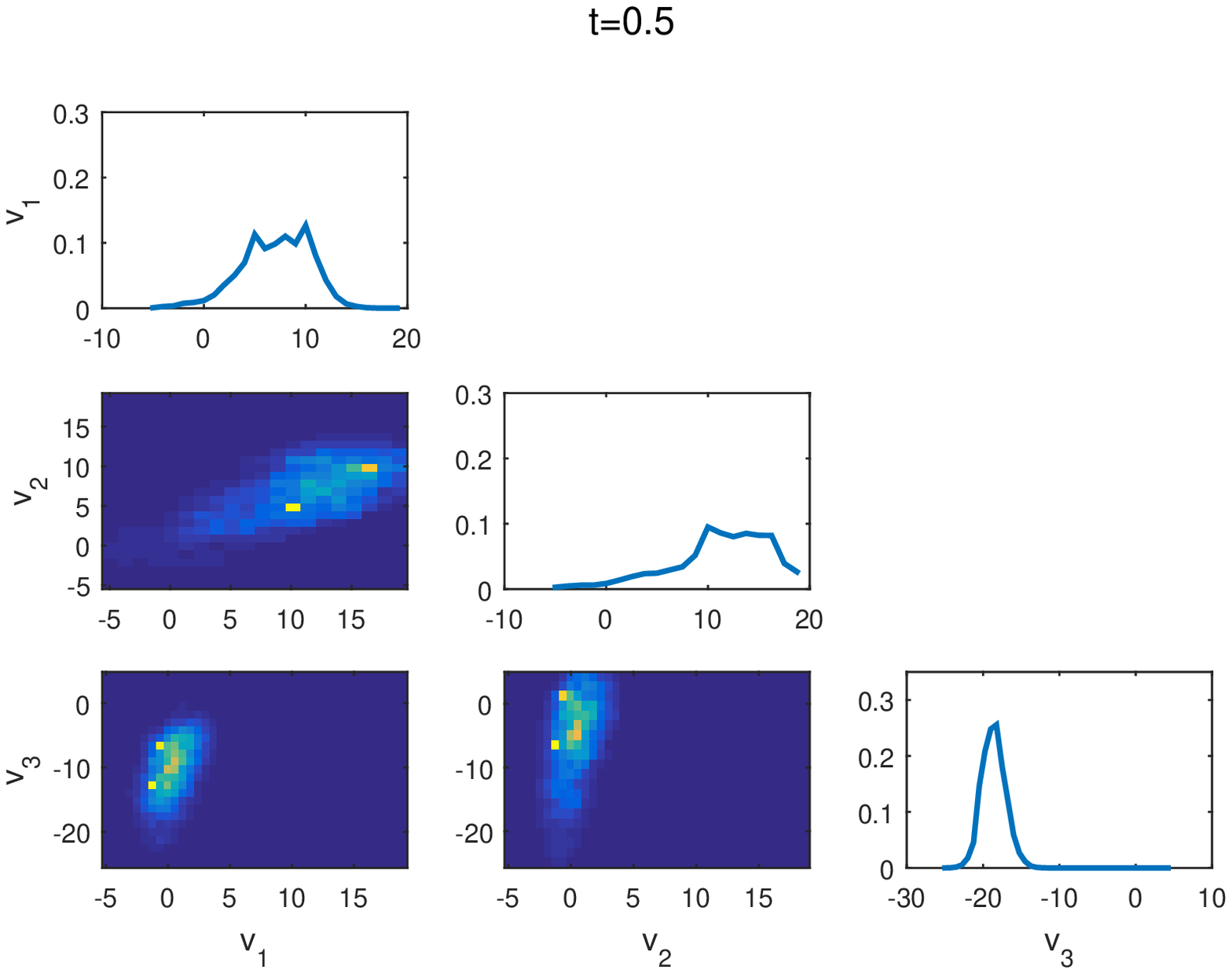}
   \includegraphics[width=0.6\textwidth]{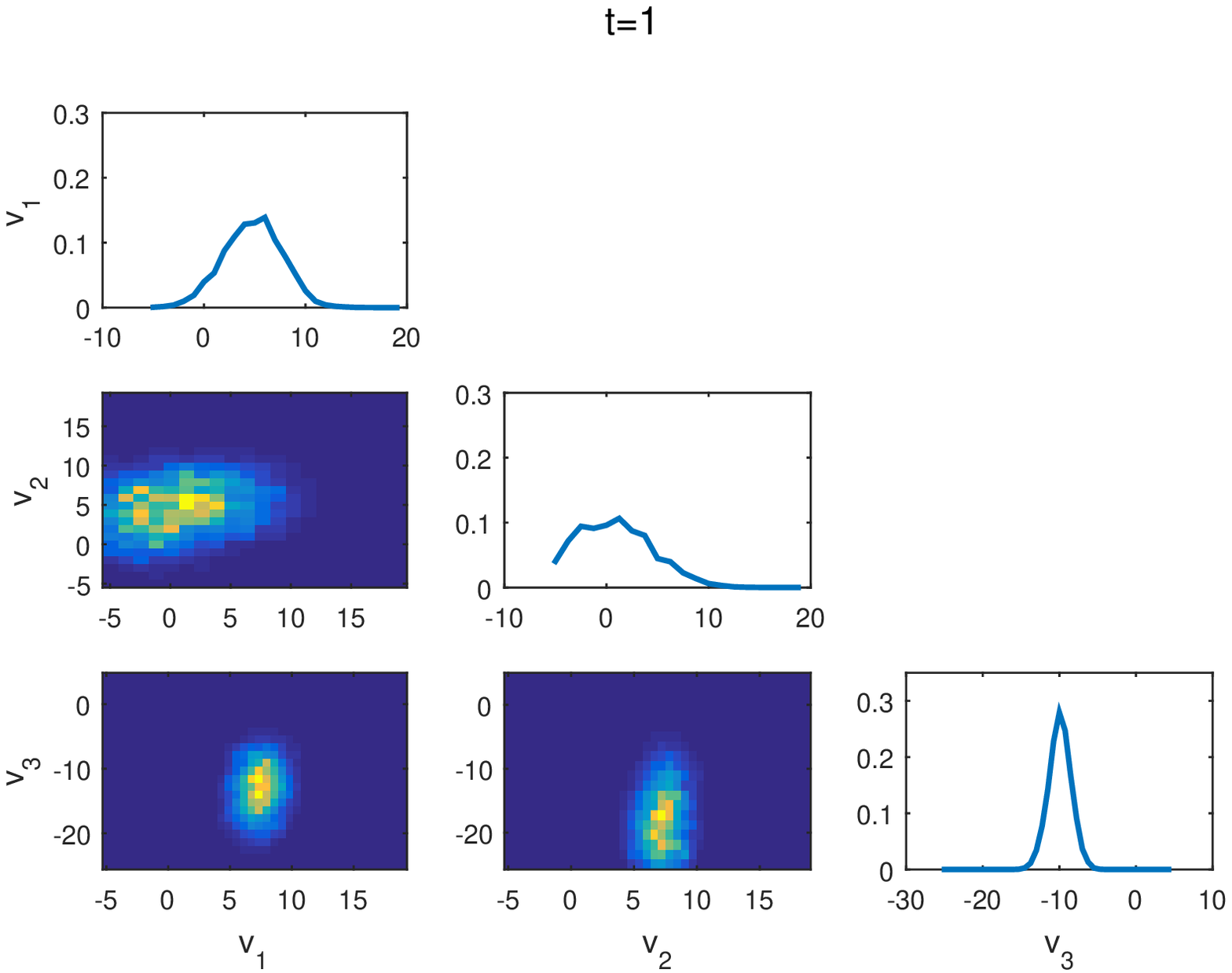}
  \caption{\footnotesize{1-D and 2-D posterior marginal probability density functions of $v$ .}}\label{grid9}
\end{figure}

The results above show that PFOF  has a higher accuracy for  state estimates than ExKF in this chaotic nonlinear system. The former can also give estimates of probability density functions to gain  more information of the state in the probabilistic sense.


\section{Conclusions}

\label{sect6}
A new filtering method was proposed to estimate filtering distribution of the state under the framework of Perron-Frobenius operator. We formulated filtering problems for discrete and continuous stochastic dynamical systems and applied the Perron-Frobenius operator to  propagation of the posterior probability density function. The finite-dimensional approximation of the PFO  was realized by Ulam's method, which provides a Galerkin projection space spanned by indicator functions.   With Ulam's method, the posterior PDF was discretized and expressed by the weights of basis functions. Then the evolution of PDF became the transition of the weights vectors, which were iterated by PFO and likelihood function. This procedure was called Perron Frobenius operator filter. Thus, the empirical PDF was determined  by a convex combination of indicator functions. We gave an error estimate of the proposed method and proved that its accuracy is higher than that of particle filters. Furthermore, a low-rank Perron-Frobenius operator filter was presented to approximate density functions via spectral decomposition. The decomposition was realized by eigendecomposition of discretized PFO.  Finally, the proposed  method was implemented for three stochastic filtering problems, which included  a linear discrete system,  a nonlinear discrete system and  a nonlinear continuous chaotic system. The numerical  results showed that the proposed method has better accuracy and better robustness compared with particle filters and ExKF.

%

\smallskip
\bigskip
\textbf{Acknowledgement:}
L. Jiang acknowledges the support of NSFC 12271408 and the
Fundamental Research Funds for the Central Universities.


\end{document}